\documentclass{article}

\usepackage{color} 
\usepackage{graphicx} 
\usepackage{xspace}
\usepackage{amsmath} 
\usepackage{amsfonts} 
\usepackage{amssymb}
\usepackage{hyperref} 
\usepackage{amsbsy} 
\usepackage{ulem}
\usepackage{tkz-euclide}
\usepackage[percent]{overpic} 

\usepackage{apxproof}

\newcommand{\footremember}[2]{%
    \footnote{#2}
    \newcounter{#1}
    \setcounter{#1}{\value{footnote}}%
}
\newcommand{\footrecall}[1]{%
    \footnotemark[\value{#1}]%
}

\newtheorem{lem}{Lemma}
\newtheorem{thm}{Theorem}

\newtheorem{defn}{Definition}

\newcommand{\red}[1]{\textcolor{red}{#1}}

\newcommand{\secref}[1]{Section~\ref{#1}}
\newcommand{\figref}[1]{Figure~\ref{#1}}
\newcommand{\defref}[1]{Definition~\ref{#1}}

\newcommand{\tabref}[1]{Table~\ref{#1}}
\newcommand{\thmref}[1]{Theorem~\ref{#1}}
\newcommand{\lemref}[1]{Lemma~\ref{#1}}

\newcommand{\reals}{{\mathbb{R}}\xspace} 
\newcommand{\mcP}{{\mathbb{P}}\xspace} 

\newcommand{\eg}{{\it e.g.}\xspace}          
\newcommand{\cf}{{\it cf.}\xspace}           

\newcommand{\bChar}[1]{\textbf{{#1}}} 
\newcommand{\bSym}[1]{\boldsymbol{{#1}}} 

\newcommand{\Time}{<\text{Time}>}

\title{An $H^1$-Conforming Solenoidal Basis for Velocity Computation on Powell-Sabin Splits for the Stokes Problem}
\author{Jeffrey M. Connors\footremember{uconn}{University of Connecticut, 341 Mansfield Road, Storrs, CT, USA} \and Michael Gaiewski\footrecall{uconn} \footnote{Corresponding author.  E-mail: michael.gaiewski@uconn.edu}}
\date{}

\begin{document}

\maketitle

\abstract{A solenoidal basis is constructed to compute velocities using a certain finite element method for the Stokes problem.  The method is conforming, with piecewise linear velocity and piecewise constant pressure on the Powell-Sabin split of a triangulation.  Inhomogeneous Dirichlet conditions are supported by constructing an interpolating operator into the solenoidal velocity space.  The solenoidal basis reduces the problem size and eliminates the pressure variable from the linear system for the velocity.  A basis of the pressure space is also constructed that can be used to compute the pressure after the velocity, if it is desired to compute the pressure.  All basis functions have local support and lead to sparse linear systems.  The basis construction is confirmed through rigorous analysis.  Velocity and pressure system matrices are both symmetric, positive definite, which can be exploited to solve their corresponding linear systems.  Significant efficiency gains over the usual saddle-point formulation are demonstrated computationally.}

\section{Introduction}\label{sec:intro} 
This paper relates to finite element computations for the incompressible Stokes problem in two dimensions.  
Given the real, simply connected and polygonal domain $\Omega \subset \reals^2$, we consider the Dirichlet problem to solve for velocity $\bChar{u} :\overline{\Omega}\to \reals^2$ and pressure $p:\Omega \to \reals$ such that 
\begin{align}
    -\nu \Delta \bChar{u} +\nabla p &= \bChar{f} \quad \text{on} \ \Omega , \label{eqn:pde} \\ 
    \nabla \cdot \bChar{u} &= 0 \quad \text{on} \ \Omega , \label{eqn:incompressible} \\ 
    \bChar{u} &= \bChar{g} \quad \text{on} \ \partial \Omega , \label{eqn:bcs} \\
    \text{and} \ \int_\Omega p \, d \mathbf{x} &= 0 ,
    \label{eqn:mv}
\end{align}
where $\nu>0$ is a constant viscosity parameter and $\bChar{g}$ is a target set of boundary values for the velocity field.  
Bold font will be reserved for vectors and spaces of vector-valued functions.  

There are countless research expositions related to the Stokes problem since it connects to many scientific models and problems in mathematics.  
Here, we are primarily interested in the computation of velocity and pressure variables using a certain finite element method where the incompressibility constraint~\eqref{eqn:incompressible} is ultimately satisfied pointwise over the domain, yielding a true solenoidal velocity field.  
In contrast, many methods only satisfy this condition in some weak (integrated) sense, and special mixed pairings of elements for the velocity and pressure are generally needed for strong incompressibility.  
A motivation is that solenoidal velocities break a coupling in the consistency error between velocity and pressure variables, so that velocity computation is not polluted by errors that should only affect the pressure accuracy.  
In the literature, this property is called {\it pressure robustness}.  
A review of methods may be found in~\cite{SIAMRev2017}.  

These methods turn out to carry an additional potential benefit: to construct a locally-supported and solenoidal basis directly for the velocity space.  
Whereas a certain saddle-point problem is typically solved for the velocity and pressure variables, a solenoidal basis can be used to express the linear system in a block-triangular fashion that allows the velocity to be computed without computing the pressure.  If the pressure is desired, it can then be calculated after via a separate solve, but in both cases the saddle-point solve is replaced by smaller, symmetric positive-definite systems.  
Few methods of this type exist at present, and it is not clear that a solenoidal basis can always have a local support.  
There are methods of discontinuous-Galerkin (DG) type~\cite{HL2008,MFH2008,MX2017}.  Also, for Raviart-Thomas (RT), Brezzi-Douglass-Marini (BDM) and hybridized locally-DG mixed elements that possess a weak divergence but not a weak gradient, see~\cite{CCS2006,WWY2009,RCSR2018,YE2021}.  
In the $H^1$-conforming case there is method with fourth-order polynomial velocities~\cite{ParkARXIV} and a method with first-order velocities~\cite{QS2007}.  

The purpose of this paper is to develop the decoupled velocity and pressure computations for the mixed pair in~\cite{Fabien_2022}.  
The method enforces solenoidal, first-order polynomial velocities and is $H^1$-conforming.  
The element order is the same as the method in~\cite{QS2007}, but the finite element mesh is quite different.  
We use Powell-Sabin splits that subdivide triangles of a fairly general mesh into six subtriangles (detailed later), whereas the meshes of~\cite{QS2007} use rectangular meshes and subdivide each rectangle into four triangles.  
This latter meshing approach may be less convenient for some applications.  
Besides the geometry, the method of this paper allows for a local macro-element assembly requiring only the six local triangles of the Powell-Sabin split at once, grouping nodal basis functions for four nodes.   
The method of~\cite{QS2007} refers to macro-element basis functions over a nine-rectangle grid of four triangles per rectangle, hence thirty-six triangles, and ultimately groups thirteen nodal functions together.  

We summarize the paper contents as follows.  
Details of the Powell-Sabin mesh and finite element method for Stokes are given in~\secref{sec:femStokes}, along with some preliminary technical lemmas.  
This includes a discussion about handling Dirichlet boundary conditions.  
In~\secref{sec:divFree} we construct a solenoidal basis for the finite element velocity space with local support that can accommodate Dirichlet conditions.  
The properties are proved, and we include details of the local construction for implementation.  
In~\secref{sec:pressure} we construct a local basis of the pressure space such that each basis function is the divergence of a known basis function in the (non-solenoidal) discrete velocity space, with proof.  
This pressure basis requires a subset of vertices to be marked using a graph-theoretic {\it spanning tree}, for which purpose there is already an inexpensive algorithm due to Kruskal~\cite{kruskal1956}.  
The pressure basis is fairly simple to implement once this is done.  
Computational examples are given in~\secref{sec:comps} comparing the velocity and pressure computations to the classical saddle-point system, the latter using the usual non-solenoidal basis for velocity.  
A summary discussion is provided in~\secref{sec:summary}.

\section{A finite element method on Powell-Sabin splits}\label{sec:femStokes} 
This section provides some mathematical preliminaries and then outlines the finite element spaces studied, focusing on a specific instance from amongst those discussed in~\cite{Fabien_2022}.  

\subsection{Preliminaries}\label{sec:prelims} 
Given any countable set $S$, we denote by $|S|$ its cardinality, whereas for sets $S\subset \Omega$, $|S|$ denotes the geometric area.  
Also, for $S\subset \Omega$ we refer to the function spaces
\begin{align*}
    L^2 (S) &:= \left\{ f:S\to \reals \, | \, \| f\|^2 = \int_S |f|^2\, d\bChar{x} <\infty \right\} \\ 
       L^2_0 (S) &:= \left\{ f \in L^2 (S) \, | \, \int_S f \, d\bChar{x} =0 \right\}  \\ 
       H^k (S) &:= \left\{ f:S\to \reals \, | \, \| f\|_k^2 = \sum_{|\alpha|\leq k} \int_S |D^{\alpha} f|^2\, d\bChar{x} <\infty \right\} \\ 
              H^k_0 (S) &:= \left\{ f\in H^k (S) \, | \, D^{\alpha}f|_{\partial \Omega}=0,\ \forall \alpha \, \text{with}\, |\alpha| < k \right\} 
\end{align*}
Here, $\bChar{x} = (x,y)$ are coordinates in $S\subset\Omega$.  
Generally, vectors and their components have the form $\bChar{f} = (f_1 , f_2)$.  
Vector-valued spaces use bold font as well: 
\[
\bChar{L}^2 (S) = \left[ L^2(S)\right]^2,\ \bChar{H}^k (S) = \left[ H^k (S) \right]^2 , \ \text{and} \ 
\bChar{H}^k_0 (S) = \left[ H^k_0 (S) \right]^2 .
\] 
In order to treat the Dirichlet boundary conditions~\eqref{eqn:bcs}, let $\gamma_0 :\bChar{H}^1 (\Omega)\to \bChar{H}^{1/2}(\partial \Omega)$ be the classical trace operator. 
We also use the divergence-free subspaces 
\begin{align*}
       \bChar{V} (S) &:= \left\{ \bChar{v} \in \bChar{H}^1 (S) \, | \, \nabla \cdot \bChar{v} =0 \, \text{on}\, \Omega \right\}  \\ 
             \bChar{V}_0 (S) &:= \left\{ \bChar{v} \in \bChar{H}^1_0 (S) \, | \, \nabla \cdot \bChar{v} =0 \, \text{on}\, \Omega \right\} .
\end{align*}

Dot-product and inner-product notations:  
\begin{align*}
\bChar{f} \cdot \bChar{g} &= f_1g_1 +f_2g_2, \ \forall \bChar{f}\in \reals^2,\ \bChar{g}\in \reals^2 \\ 
    \left( f,g \right) &= \int_\Omega f(\bChar{x})g(\bChar{x})\, d\bChar{x},\ \forall f\in L^2 (\Omega) ,\ g\in L^2 (\Omega) \\ 
    \left( \bChar{f},\bChar{g} \right) &= \int_\Omega \bChar{f}(\bChar{x})\cdot \bChar{g}(\bChar{x})\, d\bChar{x},\ \forall \bChar{f}\in \bChar{L}^2 (\Omega) ,\ \bChar{g}\in \bChar{L}^2 (\Omega) \\ 
    \nabla \bChar{f}(\bChar{x}) : \nabla \bChar{g}(\bChar{x}) &= \sum_{i=1,2} \nabla f_i (\bChar{x}) \cdot \nabla g_i(\bChar{x}),\ \forall \bChar{f}\in \bChar{H}^1 (\Omega) ,\ \bChar{g}\in \bChar{H}^1 (\Omega) \\ 
    \left( \nabla \bChar{f},\nabla \bChar{g} \right) &= \int_\Omega \nabla \bChar{f}(\bChar{x}):\nabla \bChar{g}(\bChar{x})\, d\bChar{x},\ \forall \bChar{f}\in \bChar{H}^1 (\Omega) ,\ \bChar{g}\in \bChar{H}^1 (\Omega) .
\end{align*}


\subsection{Finite element method}\label{sec:FEM} 
Let $\tau_h^M = \{ T_k^M \}$ be a family of conforming triangulations of $\Omega$, denoting the triangular mesh elements by $T_k$ for $1\leq k \leq |\tau_h^M|$.  
Here, $h=\max_k \{ h_k \}$ where $h_k=\text{diam} (T_k^M)$ is the maximum diameter of a triangle in the mesh.  
The superscript $M$ is used to distinguish this {\it macro-mesh} from its Powell-Sabin split, the latter denoted here by $\tau_h$.  
The mesh $\tau_h$ is formed via a subdivision of each triangle $T_k^M$ into six 
smaller triangles.  

The Powell-Sabin mesh $\tau_h$ is constructed via the following procedure.  
Find the incenter of each $T_k^M$.  
This is connected via line segments to the vertices of $T_k^M$.  
Given any two adjacent triangles $T_k^M$ and $T_j^M$, their incenters are then connected via a line segment, which necessarily intersects their shared edge.  
If a triangle $T_k^M$ has an exterior edge (on $\partial \Omega$), the midpoint 
of the exterior edge is connected via a line segment to the incenter.  
An important feature of the Powell-Sabin split is the creation of {\it singular points} in the mesh.  
A singular point is any vertex such that all edges meeting at that vertex fall on precisely two straight lines.  
By construction, the mesh $\tau_h$ has precisely one singular vertex along each edge of $\tau_h^M$, and no others. 
 
Given a vertex $\bChar{z}$, it is convenient for us to use a local indexing for the set of triangles that share this vertex.  Such a set is denoted by 
\[
\tau^M_{\bChar{z}} = \left\{ T^M_1 , \ldots , T^M_{|\tau^M_{\bChar{z}}|}\right\} \subset \tau_h^M 
\]
for triangles in the macro-mesh, or else 
\[
\tau_{\bChar{z}} = \left\{ T_1 , \ldots , T_{|\tau_{\bChar{z}}|}\right\} \subset \tau_h
\] 
for the Powell-Sabin mesh.  
Throughout this paper we assume an ordering of the local triangles so that $T^M_j$ and $T^M_{j+1}$ (or $T_j$ and $T_{j+1}$) share an edge.  
Moreover, we assume that a point rotated counter-clockwise around $\bChar{z}$ that crosses the shared edge will move out of triangle $j$ and into $j+1$.  
 Graphically, the Powell-Sabin splitting is summarized in~\figref{fig:PSmesh}, along with an example of the set $\tau_{\bChar{z}}$ with $\bChar{z}$ an incenter of a triangle in $\tau^M_h$.  

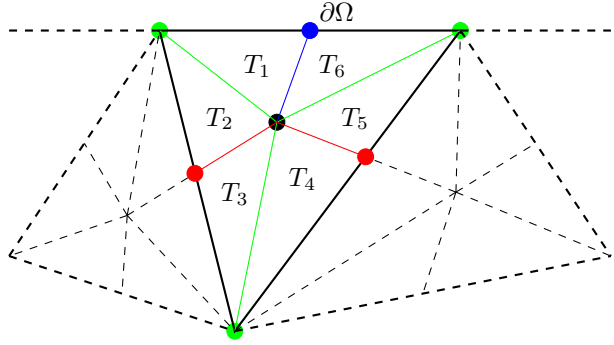
\begin{figure}
\centering 
\begin{tikzpicture}
\coordinate (A) at (0,1);
\coordinate (B) at (3,0);
\filldraw [green] (B) circle (3pt); 
\coordinate (C) at (2,4);
\filldraw [green] (C) circle (3pt); 
\coordinate (D) at (1.57,1.54);
\coordinate (AB) at (1.5,0.5); 
\coordinate (AC) at (1,2.5); 
\draw [thick, dashed] (A) -- (B);
\draw [thick, dashed] (A) -- (C);
\draw [dashed] (D) -- (A);
\draw [dashed] (D) -- (B);
\draw [dashed] (D) -- (C);
\coordinate (E) at (6,4);
\filldraw [green] (E) circle (3pt); 
\coordinate (F) at (3.56,2.78);
\filldraw [black] (F) circle (3pt); 
\draw [thick] (B) -- (E) -- (C) -- cycle;
\draw [green] (E) -- (F);
\draw [green] (B) -- (F);
\draw [green] (C) -- (F);
\coordinate (G) at (8,1);
\coordinate (H) at (5.94,1.85);
\coordinate (BE) at (4.74,2.32);
\coordinate (CE) at (4,4);
\draw [thick, dashed] (E) -- (G);
\draw [thick, dashed] (B) -- (G);
\draw [dashed] (H) -- (G);
\draw [dashed] (H) -- (B);
\draw [dashed] (H) -- (E);
\coordinate (BG) at (5.5,0.5);
\coordinate (EG) at (7,2.5);

\coordinate (BC) at (2.47,2.1);
\draw [red] (F) -- (BC);
\draw [red] (F) -- (BE);
\draw [blue] (F) -- (CE);
\draw [dashed] (D) -- (BC);
\draw [dashed] (D) -- (AB);
\draw [dashed] (D) -- (AC);
\draw [dashed] (H) -- (BG); 
\draw [dashed] (H) -- (EG); 
\draw [dashed] (H) -- (BE);
\tkzLabelPoint[above right](CE){$\partial \Omega$}
\coordinate (P) at (0,4);
\coordinate (Q) at (8,4); 
\draw [thick, dashed] (P) -- (C); 
\draw [thick, dashed] (E) -- (Q); 
\filldraw [red] (BC) circle (3pt); 
\filldraw [blue] (CE) circle (3pt); 
\filldraw [red] (BE) circle (3pt); 
\coordinate (Z1) at (3.3,3.8);
\tkzLabelPoint[below](Z1){$T_1$} 
\coordinate (Z2) at (2.8,3.1);
\tkzLabelPoint[below](Z2){$T_2$} 
\coordinate (Z3) at (3.0,2.1);
\tkzLabelPoint[below](Z3){$T_3$} 
\coordinate (Z4) at (3.9,2.3);
\tkzLabelPoint[below](Z4){$T_4$} 
\coordinate (Z5) at (4.6,3.1);
\tkzLabelPoint[below](Z5){$T_5$} 
\coordinate (Z6) at (4.3,3.8);
\tkzLabelPoint[below](Z6){$T_6$}
\end{tikzpicture}
\caption{The Powell-Sabin split illustrated for a triangle of the macro-mesh $\tau_h^M$.  The incenter $\bChar{z}$ (black dot) is connected to the vertices (green dots, edges) and neighboring incenters (red edges).  In case of an edge on the boundary, the incenter is connected to the midpoint of that edge (blue edge, dot).  The split results in six triangles, $\tau_{\bChar{z}}=\{ T_1 , \ldots , T_6 \}$, and a set of singular vertices in correspondence with the edges of $\tau_h^M$ (red and blue dots).}\label{fig:PSmesh}
\end{figure}

Some notation of use follows here, where an edge or vertex is called {\it exterior} if it lies on $\partial \Omega$, otherwise it is {\it interior}.  
\begin{itemize}
\item $E_{int}^M$ (and $E_{ext}^M$) -- interior (and exterior) edges of $\tau_h^M$ 
\item $V_{int}^M$ (and $V_{ext}^M$) -- interior (and exterior) vertices of $\tau_h^M$ 
\item $S_{int}$ (and $S_{ext}$) -- interior (and exterior) singular vertices of $\tau_h$
\item $E^M = E_{int}^M \cup E_{ext}^M$, $V^M = V_{int}^M \cup V_{ext}^M$, $S = S_{int} \cup S_{ext}$, 
\end{itemize} 
Note the cardinality relationships 
\begin{equation}
    |E_{int}^M| = |S_{int}| \quad \text{and} \quad |E_{ext}^M| = |S_{ext}| . \label{eqn:card}
\end{equation}

Given any region $S\subset \Omega$, let $\mcP_k (S)$ be the space of polynomials of order $k$ or less on $S$, with corresponding vector-valued space $\bSym{\mcP}_k (S) = [\mcP_k (S)]^2$.  Spaces needed to discuss the discrete velocity are  
\begin{align*}
    \bChar{X}_h &= \left\{ \bChar{v}_h \in [\mathcal{C} (\overline{\Omega})]^2 \ | \ \bChar{v}_h \in \bSym{\mcP}_1 (T), \ \forall T\in \tau_h \right\} \subset \bChar{H}^1 (\Omega) \\
    \bChar{X}_h^0 &= \left\{ \bChar{v}_h \in \bChar{X}_h \ | \ \bChar{v}_h =0\ \text{on}\ \partial \Omega \right\} \subset \bChar{H}^1_0 (\Omega) \\ 
    \bChar{V}_h &= \bChar{X}_h \cap \bChar{V} (\Omega) \\ 
        \bChar{V}_h^0 &= \bChar{V}_h \cap \bChar{V}_0 (\Omega)
\end{align*}

The corresponding pressure space is identified after observing a key property of the divergence of the discrete velocity (\cf~\cite{Fabien_2022}; see~\cite{SV1985} for a proof).  
\begin{lem}\label{lem:discDiv}
Let $q\in \mathcal{C}(T)$ for all $T\in \tau_h$ and define a function 
\begin{equation*}
    \theta_{\bChar{z}} (q) := \left\{ \begin{array}{ll} 
    q|_{T_1} (\bChar{z}) - q|_{T_2} (\bChar{z}) + q|_{T_3} (\bChar{z}) - q|_{T_4} (\bChar{z}), & \ \text{if} \ \bChar{z}\in S_{int} ,  \\ 
     q|_{T_1} (\bChar{z}) - q|_{T_2} (\bChar{z}) , & \ \text{if} \ \bChar{z}\in S_{ext} , 
    \end{array}\right. 
\end{equation*}
where $\tau_{\bChar{z}} = \{ T_1 , \ldots , T_4 \}$ if $\bChar{z}\in S_{int}$ or $\tau_{\bChar{z}} = \{ T_1 , T_2 \}$ if $\bChar{z}\in S_{ext}$.  
Then for any $\bChar{v}_h\in \bChar{X}_h^0$ and singular vertex $\bChar{z}\in S$ there holds $\theta_{\bChar{z}} (\nabla \cdot \bChar{v}_h )=0$.  
\end{lem} 
The pressure space is 
\[
P_h = \left\{ q_h \in L^2_0 (\Omega) \ | \ q_h|_T \in \mcP_0 (T),\, \forall T\in \tau_h,\, \theta_\bChar{z} (q_h) = 0, \, \forall \bChar{z}\in S \right\} . 
\]

Before the finite element problem can be stated formally, we must address a critical property of the discrete, solenoidal velocity space $\bChar{V}_h$.  This affects the boundary conditions and is also referenced later to construct a solenoidal basis in~\secref{sec:divFree}.  First, we require a technical lemma. 
\begin{lem}\label{lem:divFormula}
Let $T\subset\reals^2$ be a triangle with vertices $\bChar{z}_i$, $1\leq i\leq 3$, ordered  counter-clockwise going from $\bChar{z}_i$ to $\bChar{z}_{i+1}$ along $\partial T$.  
Each triangle edge, $e_i$, lies opposite the vertex $\bChar{z}_i$, with outward unit normal vector $\bChar{n}_i$.  Note the identity 
\begin{equation}
    \sum_{1\leq i\leq 3} |e_i| \bChar{n}_i = \bSym{0}.  
    \label{eqn:normalIdentity} 
\end{equation}
Given any $\bChar{v}\in \bChar{P}_1 (T)$, it holds that 
\begin{equation}
    \int_T \nabla \cdot \bChar{v} \, d\bChar{x} 
    = -\frac{1}{2}\sum_{1\leq i\leq 3} |e_i| \bChar{v} (\bChar{z}_i)\cdot \bChar{n}_i . 
    \label{eqn:divFormula} 
\end{equation}
\end{lem} 
\begin{proof}
Let $\mathcal{A}:\reals^2\to\reals^2$ be the linear operator that rotates a vector clockwise $90$ degrees.  Then 
\begin{equation*} 
\sum_{1\leq i\leq 3} |e_i| \bChar{n}_i 
=|e_1| \mathcal{A} \left( \frac{\bChar{z}_3 - \bChar{z}_2}{|e_1|}\right)
+|e_2| \mathcal{A} \left( \frac{\bChar{z}_1 - \bChar{z}_3}{|e_2|}\right)
+|e_3| \mathcal{A} \left( \frac{\bChar{z}_2 - \bChar{z}_1}{|e_3|}\right) 
= \bChar{0}.
\end{equation*} 
Then~\eqref{eqn:divFormula} follows from~\eqref{eqn:normalIdentity} after applying the divergence form of Green's Theorem.  
\end{proof}

We will use $ds$ for a measure along edges.   
\begin{lem}\label{lem:dofs}
Let $e^M \in E^M$ have endpoints $\bChar{z}_{1}$ and $\bChar{z}_{2}$.  
The trace of any $\bChar{v}_h\in \bChar{V}_h$ along $e^M$ may be uniquely determined from the values 
\begin{equation}
    \bChar{v}_h (\bChar{z}_1 ) , \ 
     \bChar{v}_h (\bChar{z}_2 ), \ \text{and} \ 
     \int_{e^M} \bChar{v}_h\cdot \bChar{n} \, ds , 
     \label{eqn:dofs} 
\end{equation}
where $\bChar{n}$ is a unit normal along $e^M$.   
\end{lem}
\begin{proof}
Refer to~\figref{fig:labels} to visualize the vertices, edges and normal vectors used in this proof.  
Let $\bChar{z}_s\in S$ be the singular vertex of $\tau_h$ on $e^M$.  
    Since $\bChar{v}_h$ is piecewise linear and continuous along $e^M$, it suffices to show that if the values in~\eqref{eqn:dofs} are all zero, then also $\bChar{v}_h (\bChar{z}_s)=\bChar{0}$.  
    Direct integration shows that if $\bChar{v}_h (\bChar{z}_i )=\bChar{0}$, $i=1,2$, then 
    \begin{equation}
 \int_{e^M} \bChar{v}_h\cdot \bChar{n} \, ds =0\Rightarrow 
 \bChar{v}_h (\bChar{z}_s)\cdot \bChar{n} = 0.  
 \label{eqn:l3e0}
    \end{equation}
    
    It remains only to show that the tangential component is zero; equivalently $\bChar{v}_h (\bChar{z}_s)\cdot (\bChar{z}_2 -\bChar{z}_1)= 0$.  
    Let $\bChar{z}_c$ be the incenter of a triangle in $\tau_h^M$ that has edge $e^M$.  There are two triangles, say $T_1$ and $T_2$, in $\tau_{\bChar{z}_c}$ with edges $e_1$ and $e_2$, respectively, that lie on $e^M$, opposite $\bChar{z}_c$.  They have a shared edge lying opposite $\bChar{z}_i$ on $T_i$.  Let the edge of $T_i$ opposite $\bChar{z}_s$ be $e_i^s$ and have outward unit normal $\bChar{n}_i^s$ for $i=1,2$.    
   
    Since $\bChar{v}_h$ is solenoidal, we may apply~\eqref{eqn:divFormula} on each triangle $T_i$ to find 
    \begin{align}
        |e_1^s| \bChar{v}_h (\bChar{z}_s)\cdot \bChar{n}_1^s 
        +|e_1| \bChar{v}_h (\bChar{z}_c)\cdot \bChar{n} &=0 \label{eqn:l3e1} \\
  \text{and}\quad  |e_2^s| \bChar{v}_h (\bChar{z}_s)\cdot \bChar{n}_2^s 
        +|e_2| \bChar{v}_h (\bChar{z}_c)\cdot \bChar{n} &=0 . \label{eqn:l3e2}
    \end{align}
     
    Note that edges $e_1^s$, $e_2^s$ and $e^M$ form a perimeter of triangle $T_1\cup T_2$, hence the vectors $\bChar{n}_1^s$ and $\bChar{n}_2^s$ are equivalent to a certain basis of $\reals^2$.  
    Then~\eqref{eqn:l3e1}-\eqref{eqn:l3e2} show that $\bChar{v}_h (\bChar{z}_s)=\bChar{0}$ if $\bChar{v}_h (\bChar{z}_c)\cdot \bChar{n}=\bChar{0}$.  
    Indeed, now sum~\eqref{eqn:l3e1}-\eqref{eqn:l3e2} and apply~\eqref{eqn:normalIdentity},~\eqref{eqn:l3e0} to show that 
    \begin{multline*}
        (|e_1|+|e_2|)\bChar{v}_h (\bChar{z}_c)\cdot \bChar{n} 
        = -|e_1^s| \bChar{v}_h (\bChar{z}_s)\cdot \bChar{n}_1^s 
        -|e_2^s| \bChar{v}_h (\bChar{z}_s)\cdot \bChar{n}_2^s \\ 
        = (|e_1|+|e_2|)\bChar{v}_h (\bChar{z}_s)\cdot \bChar{n} =\bChar{0}.
    \end{multline*}
\end{proof}

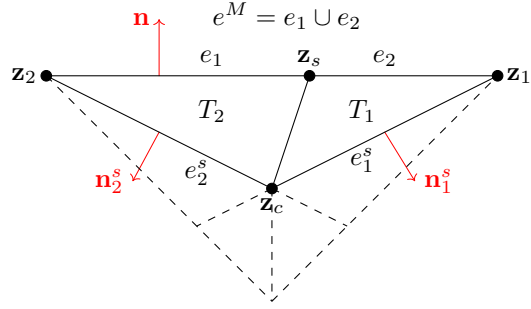
\begin{figure}
\centering 
\begin{tikzpicture}
\coordinate (A) at (0,3);
\tkzLabelPoint[above, left](A){$\bChar{z}_2$} 
\coordinate (B) at (3,0);
\coordinate (C) at (6,3);
\tkzLabelPoint[above, right](C){$\bChar{z}_1$} 
\coordinate (D) at (3,1.5);
\tkzLabelPoint[below](D){$\bChar{z}_c$} 
\coordinate (AB) at (2,1);
\coordinate (BC) at (4,1);
\coordinate (AC) at (3.5,3);
\tkzLabelPoint[above](AC){$\bChar{z}_s$} 
\coordinate (eM) at (3.2,3.5);
\tkzLabelPoint[above](eM){$e^M = e_1\cup e_2$} 
\coordinate (e1) at (2.2,3);
\tkzLabelPoint[above](e1){$e_1$} 
\coordinate (e2) at (4.5,3);
\tkzLabelPoint[above](e2){$e_2$} 
\coordinate (T1) at (4.2,2.8);
\tkzLabelPoint[below](T1){$T_1$} 
\coordinate (T2) at (2.2,2.8);
\tkzLabelPoint[below](T2){$T_2$} 
\coordinate (nt) at (1.5,3);
\coordinate (nh) at (1.5,3.75);
\tkzDrawSegment[->, red](nt,nh);
\tkzLabelPoint[above, left](nh){\red{$\bChar{n}$}} 
\coordinate (n1t) at (4.5,2.25);
\coordinate (n1h) at (4.9,1.6);
\tkzDrawSegment[->, red](n1t,n1h);
\tkzLabelPoint[right](n1h){\red{$\bChar{n}_1^s$}} 
\coordinate (e1s) at (4.2,1.6);
\tkzLabelPoint[above](e1s){$e_1^s$} 
\coordinate (n2t) at (1.5,2.25);
\coordinate (n2h) at (1.15,1.6);
\tkzDrawSegment[->, red](n2t,n2h);
\tkzLabelPoint[left](n2h){\red{$\bChar{n}_2^s$}} 
\coordinate (e2s) at (2.0,2.0);
\tkzLabelPoint[below](e2s){$e_2^s$} 
\filldraw [black] (A) circle (2pt); 
\filldraw [black] (AC) circle (2pt); 
\filldraw [black] (C) circle (2pt); 
\filldraw [black] (D) circle (2pt); 
\draw [black] (A) -- (C);
\draw [black, dashed] (A) -- (B);
\draw [black, dashed] (B) -- (C);
\draw [black] (D) -- (AC);
\draw [black] (D) -- (A);
\draw [black] (D) -- (C);
\draw [black, dashed] (D) -- (AB);
\draw [black, dashed] (D) -- (BC);
\draw [black, dashed] (D) -- (B);
\end{tikzpicture}
\caption{Labels referenced in the proof of~\lemref{lem:dofs}.}\label{fig:labels}
\end{figure}

A consequence of~\lemref{lem:dofs} is that we cannot generally interpolate boundary data $\bChar{g}\in [\mathcal{C}(\partial\Omega)]^2$ to compute a solenoidal Stokes velocity $\bChar{u}_h \in \bChar{V}_h$ by setting $\bChar{u}_h (\bChar{z}) = \bChar{g} (\bChar{z})$ at all the exterior nodes $\bChar{z}$ of the Powell-Sabin mesh.  
The following definition provides an appropriate interpolant, where the function $\bChar{G}_h$ will be explicitly constructed in~\secref{sec:divFree}.  
\begin{defn}[Boundary data interpolation]\label{defn:gh} 
Let $\bChar{g}\in \bChar{H}^{1/2}(\partial\Omega)\cap [ \mathcal{C}(\partial\Omega)]^2$.  
We denote by $\bChar{g}_h$ the interpolant of $\bChar{g}$, where $\bChar{g}_h = \gamma_0 (\bChar{G}_h)$ for some $\bChar{G}_h\in \bChar{V}_h$ and 
\begin{align}
    \bChar{g}_h (\bChar{z}) &= \bChar{g} (\bChar{z}),\ \forall \bChar{z}\in V_{ext}^M \label{eqn:ghNodes} \\
    \text{and} \quad 
    \int_{e} \bChar{g}_h \cdot \bChar{n} \, ds 
    &= \int_{e} \bChar{g} \cdot \bChar{n} \, ds , \ \forall e\in E_{ext}^M, \label{eqn:ghMoments}  
\end{align}
where $\bChar{n}$ is the outward unit normal vector on $\partial\Omega$.  
\end{defn}

We may now state the finite element problem.  We solve for $\bChar{u}_h \in \bChar{X}_h$ and $p_h \in P_h$ by decomposing $\bChar{u}_h = \bChar{w}_h +\bChar{G}_h$, $\bChar{G}_h \in \bChar{V}_h$ as in~\defref{defn:gh} and $\bChar{w}_h \in \bChar{X}_h^0$ such that 
\begin{align}
    a(\bChar{w}_h , \bChar{v}_h ) -(p_h , \nabla \cdot \bChar{v}_h ) &= L(\bChar{v}_h),\ \forall \bChar{v}_h \in \bChar{X}_h^0 \label{eqn:fem1} \\
    (\nabla \cdot \bChar{w}_h , q_h ) &= 0 ,\ \forall q_h \in P_h \label{eqn:fem2} \\ 
    a(\bChar{w}_h , \bChar{v}_h ) &:= \nu \left( \nabla \bChar{w}_h ,\bChar{v}_h \right) \nonumber \\ 
    L(\bChar{v}_h ) &:= (\bChar{f} , \bChar{v}_h) - a (\bChar{G}_h , \bChar{v}_h) \nonumber 
\end{align}
However, it is shown in~\cite{Fabien_2022} that the solution of~\eqref{eqn:fem1}-\eqref{eqn:fem2} will in fact have $\bChar{w}_h\in \bChar{V}_h^0$.  
Thus $\bChar{u}_h \in \bChar{V}_h$, and we focus in this paper on directly solving the problem on the solenoidal subspace; find $\bChar{w}_h \in \bChar{V}_h^0$ satisfying 
\begin{equation}
       a(\bChar{w}_h , \bChar{v}_h )  = L(\bChar{v}_h),\ \forall \bChar{v}_h \in \bChar{V}_h^0 . \label{eqn:dffem}  
\end{equation}
This is accomplished in~\secref{sec:divFree}.  Note that the pressure variable need not be computed, but if it is desired we also provide a way to compute it after the velocity is found.  
One solves the problem: find $p_h \in P_h$ such that
\begin{equation}
(p_h , \nabla \cdot \bChar{v}_h ) =
       a(\bChar{u}_h , \bChar{v}_h )   -(\bChar{f},\bChar{v}_h),\ \forall \bChar{v}_h \in \bChar{X}_h^0 \setminus \bChar{V}_h^0. \label{eqn:dfPressure}  
\end{equation}
This method is detailed in~\secref{sec:pressure}.   

\section{A solenoidal velocity basis}\label{sec:divFree} 
Explicit constructions of basis functions are provided here.  
We also show how to implement boundary conditions, and a localized implementation of the basis using a macro-element technique.  

\subsection{Bases of $\bChar{V}_h$ and $\bChar{V}_h^0$}\label{sec:basisDefn}  
Given a vertex $\bChar{z}\in V^M$, the support of each basis function will be $\tau_{\bChar{z}}^M$.  
Consider the subspaces with this support: 
\begin{equation}
    \bChar{V}_{\bChar{z}} = \left\{ \bChar{v}_h \in \bChar{V}_h \ | \ \bChar{v}_h (\bChar{x}) = \bChar{0} \ \forall \bChar{x} \in T^M , \ \forall T^M\subset \tau_h^M \setminus 
 \tau_{\bChar{z}}^M \right\} . 
 \label{eqn:Vz} 
\end{equation}
We will first construct a basis of $\bChar{V}_{\bChar{z}}$.  
A certain union of these across vertices $\bChar{z}$ will later form the desired basis for $\bChar{V}_h$.  
\begin{lem}\label{lem:localBasis} 
Given any $\bChar{z}\in V^M$, it holds that $\text{dim}(\bChar{V}_{\bChar{z}})=3$.  
Moreover, a basis $\{\bSym{\Phi}_1,\bSym{\Phi}_2 ,\bSym{\Phi}_3 \}$ is found by taking $(\alpha_1 , \beta_1 , \delta_1)=(1,0,0)$, $(\alpha_2 , \beta_2 , \delta_2)=(0,1,0)$, $(\alpha_3 , \beta_3 , \delta_3)=(0,0,1)$, and setting 
\begin{align}
    \bSym{\Phi}_i (\bChar{z}) &= (\alpha_i , \beta_i) \label{eqn:Phi12} \\ 
    \text{and} \quad \int_{e_j^M} \bSym{\Phi}_i \cdot \bChar{n}_{j}\, ds &= \delta_i, \label{eqn:Phi3} 
\end{align}
for $1\leq i\leq 3$, where, for each $T_j^M \in \tau_{\bChar{z}}^M$, 
$\{e_{j-1}^M,e_j^M \}\subset E^M$ are its edges that have $\bChar{z}$ as one endpoint, and $\bChar{n}_j$ is the unit normal on $e_j^M$ that points in the direction of counter-clockwise rotation around $\bChar{z}$.  
\end{lem}
\begin{proof}
Note that we must have 
\[
0 = \int_{T_j^M} \nabla \cdot \bSym{\Phi}_i \, d\bChar{x} 
= \pm \left( \int_{e_j^M} \bSym{\Phi}_i \cdot \bChar{n}_{j}\, ds - \int_{e_{j-1}^M} \bSym{\Phi}_i \cdot \bChar{n}_{j-1}\, ds \right) 
\]
so that the value of $\delta_i$ is $j$-independent in~\eqref{eqn:Phi3}.  
Since $\bChar{z}$ is the common vertex of all $T_j^M\in \tau_{\bChar{z}}^M$, the conditions~\eqref{eqn:Phi12}-\eqref{eqn:Phi3} impose the same constraints on each $T_j^M$, so it suffices to look at the restriction to one $T_j^M$.  
Denote by $\bChar{z}_c$ the incenter of $T_j^M$ and $e^M$ the edge that lies opposite $\bChar{z}$.  

Consider the vector space 
\[
\bChar{W}_j = \left\{ \bChar{v}_h \in \mathcal{C} (T_j^M)\ | \ \bChar{v}_h \in  \bChar{P}_1 (T), \forall T\in\tau_{\bChar{z}_c}, \ \text{and}\ \bChar{v}_h =\bChar{0} \ \text{on}\ e^M \right\} .
\]
The six triangles in $\tau_{\bChar{z}_c}$ together have seven distinct vertices (see~\eg~\figref{fig:PSmesh}, with $\bChar{v}_h=\bChar{0}$ at the three vertices on $e^M$.  
Any $\bChar{v}_h\in \bChar{W}_j$ is uniquely determined by its values at the other four vertices, so $\text{dim}(\bChar{W}_j)=8$.  
Two triangles $T_i\in \tau_{\bChar{z}_c}$ have an edge lying on $e^M$, say $T_1$ and $T_2$ (see~\figref{fig:labels}).  
Eight linear constraints may be introduced by requiring 
\begin{align*}
    \bChar{v}_h (\bChar{z}) &=(\alpha,\beta) , \\
    \quad \int_{e_j^M} \bChar{v}_h \cdot \bChar{n}_{j}\, ds &= \delta ,
\end{align*}  
for any desired values $\alpha$, $\beta$ and $\delta$, 
and also setting $\nabla\cdot \bChar{v}_h=0$ on triangles $T_i$, $2\leq i\leq 6$.  In fact, then $\nabla\cdot \bChar{v}_h=\bChar{0}$ will hold also on $T_1$.  
This is verified as follows.  
Note that $T_1$ and $T_2$ each have two vertices on $e^M$, where $\bChar{v}_h=\bChar{0}$, and their third vertex is $\bChar{z}_c$.  
Apply~\lemref{lem:divFormula} on $T_2$ (where we impose $\nabla \cdot \bChar{v}_h=\bChar{0}$) first, which reduces to 
\[
0=\bChar{v}_h (\bChar{z}_c) \cdot \bChar{n} ,
\]
where $\bChar{n}$ is the outward unit normal for $T_j^M$ along $e^M$.  
Application of~\lemref{lem:divFormula} on $T_1$ shows that $\nabla \cdot \bChar{v}_h$ is proportional to $\bChar{v}_h (\bChar{z}_c) \cdot \bChar{n}$, thus $\nabla \cdot \bChar{v}_h=\bChar{0}$ on $T_1$.  

The proof reduces to showing that if $\alpha=\beta=\delta=0$ then $\bChar{v}_h =\bChar{0}$ at all seven vertices.  
We already have $\bChar{v}_h=\bChar{0}$ at the three vertices on $e^M$, plus at $\bChar{z}$ as well under these conditions.  
Since $\delta=0$, we may apply~\lemref{lem:dofs} to conclude that $\bChar{v}_h=\bChar{0}$ along the other two edges of the macro-triangle $T_j^M$, hence at the two singular vertices on those edges, leaving only the incenter vertex $\bChar{z}_c$.  
Above we showed $\bChar{v}_h (\bChar{z}_c) \cdot \bChar{n} =0$; repeating the arguments (apply~\lemref{lem:divFormula}) on any of the triangles $T_i$ with $3\leq i\leq 6$ shows that $\bChar{v}_h (\bChar{z}_c)\cdot \bChar{n}^* =0$ for $\bChar{n}^* =\bChar{n}_{j}$ or $\bChar{n}^* =\bChar{n}_{j-1}$.  
Since $\bChar{n}$ and $\bChar{n}^*$ cannot be parallel, it follows $\bChar{v}_h (\bChar{z}_c) =0$.  
\end{proof}

The solenoidal basis is discussed in the next result.  
It requires the exclusion of a single function to form a basis, but the association made with the boundary by choosing a vertex $\bChar{z}_0 \in V_{ext}^M$ is for convenience in implementing boundary conditions; see~\secref{sec:implement}.  
Also, in~\secref{sec:pressure} a single boundary vertex must be chosen to extract a special basis for the pressure space, so in practice it is natural to use the same vertex  $\bChar{z}_0$.  
\begin{thm}\label{thm:Vh}
Given an indexing $\bChar{z}_k \in V^M$, $0\leq k \leq |V^M|-1$ such that $\bChar{z}_0 \in V_{ext}^M$ (lies on $\partial\Omega$), we denote by 
$\bSym{\Phi^{(k)}_i}$, $1\leq i \leq 3$, the basis functions of $\bChar{V}_{\bChar{z}}$ defined in~\lemref{lem:localBasis} for $\bChar{z} = \bChar{z}_k$.  
We define sets 
\begin{align*}
    \mathcal{B} &= \left[ \bigcup_{i=1}^{3}\bigcup_{k=0}^{|V^M|-1} \bSym{\Phi^{(k)}_i}  \right] \setminus \left\{ \bSym{\Phi^{(0)}_3}  \right\} \\ 
    \mathcal{B}^0 &= \left\{ \bSym{\Phi^{(k)}_i} \in \mathcal{B}\, | \, \bChar{z}_k \in V_{int}^M \right\} .
\end{align*}
It holds that $\mathcal{B}$ is a basis of $\mathbf{V}_h$ and $\mathcal{B}^0$ is a basis of $\mathbf{V}_h^0$.  
\end{thm} 
\begin{proof}
    We prove that $\mathcal{B}$ is a basis of $\mathbf{V}_h$ by showing that $\text{dim}(\bChar{V}_h)=3|V^M|-1$ along with the linear independence of $\mathcal{B}$.  
The former claim is a counting argument.  
Since $\bChar{V}_h$ is the kernel of the divergence operator on $\bChar{X}_h$, we may use the rank-nullity theorem.  
It is proved in (~\cite{GLN2020}, Theorem 4) that $\text{dim}(\nabla \cdot \bChar{X}_h) = |E^M| +3|\tau_h^M|$.  
Thus 
\begin{multline}
\text{dim}(\bChar{V}_h)= \text{dim}(\bChar{X}_h) - |E^M| -3|\tau_h^M| 
= 2|V| - |E^M| -3|\tau_h^M| \\ 
= 2(|V^M|+|E^M|+|\tau_h^M|) - |E^M| -3|\tau_h^M| \\
= 2|V^M|+ |E^M| -|\tau_h^M| , \label{eqn:dimVh}
\end{multline} 
and we apply the Euler identity $|E^M|-|\tau_h^M| = |V^M|-1$ to conclude $\text{dim}(\bChar{V}_h)=3|V^M|-1 = |\mathcal{B}|$.  
Next, assume $\bChar{v}\in\text{span}(\mathcal{B})$ satisfies 
\[
\bChar{v} = \sum_{i=1}^3 \sum_{k=0}^{|V^M|-1} c_i^{(k)} \bSym{\Phi}_i^{(k)} =\bChar{0} = (0,0), 
\] 
where $c_3^{(0)}=0$ (recall this basis function is excluded from $\mathcal{B}$).  
By construction, we have at the vertices in $\bChar{z}_j \in V^M$ 
\[
\bChar{v}(\bChar{z}_j) = ( c_1^{(j)} ,c_2^{(j)} ) =(0,0), \ 0\leq j \leq |V^M|-1, 
\] 
and we may reduce 
\[
\bChar{v} = \sum_{k=0}^{|V^M|-1} c_3^{(k)} \bSym{\Phi}_3^{(k)} =\bChar{0} = (0,0).
\] 
Note that each coefficient $c_3^{(k)}$ may be associated with a point $\bChar{z}_k\in V^M$; these are joined by the edges in $E^M$ to form a connected graph.  
Given any edge $e\in E^M$ with unit normal $\bChar{n}$ and endpoints $\bChar{z}_n$ and $\bChar{z}_m$, we have 
\[
\int_e \bChar{v}\cdot \bChar{n}\, ds = \pm \left( c_3^{(n)} - c_3^{(m)} \right) =0, 
\] 
and therefore $c_3^{(n)}=c_3^{(m)}$ holds for all values $n$, $m$.  
Since $c_3^{(0)}=0$, we have $c_3^{(k)}=0$ for all $k$.  
This proves linear independence of $\mathcal{B}$.  
Finally, the fact that $\mathcal{B}^0$ is a basis of $\bChar{V}_h^0$ now follows by simply restricting the boundary values to zero.  
\end{proof}

\subsection{Implementation details}\label{sec:implement} 
The correct representation of Dirichlet conditions is now easy to show.  
Recall the decomposition $\bChar{u}_h = \bChar{G}_h +\bChar{w}_h$, where $\bChar{w}_h \in \bChar{V}_h^0$ solves~\eqref{eqn:dffem} and $\bChar{G}_h\in \bChar{V}_h$ satisfies~\defref{defn:gh}.  
The component $\bChar{w}_h$ is computed using the basis $\mathcal{B}^0$.  
We may expand out $\bChar{G}_h$ using $\mathcal{B}\setminus \mathcal{B}^0$; for convenience in exposition, assume the exterior edge ordering $e_k^M\in E_{ext}^M$, $1\leq k\leq |E_{ext}^M|$, such that $e_k^M$ and $e_{k+1}^M$ share a vertex.  
This supports a vertex labelling $\bChar{z}_k\in V^M$, $0\leq k \leq |V^M|-1$, such that each edge $e_k^M\in E_{ext}^M$ has endpoints $\bChar{z}_{k-1}$ and $\bChar{z}_k$.  

We may expand out 
\[
\bChar{G}_h =  \sum_{i=1}^3 \sum_{k=0}^{|V_{ext}^M|-1} c_i^{(k)} \bSym{\Phi}_i^{(k)} .
\] 
Boundary values are still interpolated at the points $\bChar{z}_k\in V_{ext}^M$, corresponding to~\eqref{eqn:ghNodes}: 
\[
\bChar{G}_h \left( \bChar{z}_k \right)  = (c_1^{(k)},c_2^{(k)} ) = \bChar{g} \left( \bChar{z}_k \right), \ 0\leq k \leq |V_{ext}^M|-1 .
\]  
At split points in the Powell-Sabin mesh along the boundary, the values of $\bChar{G}_h$ are determined implicitly through the normal moments~\eqref{eqn:ghMoments}.  
Specifically, we have the system of equations 
\[
\int_{e_k^M} \bChar{G}_h\cdot \bChar{n} \, ds 
= \pm (c_3^{(k)}-c_3^{(k-1)})
= \int_{e_k^M} \bChar{g}\cdot \bChar{n} \, ds , \ 1\leq k \leq |V_{ext}^M|-1 , 
\]
where $\bChar{n}$ is the outward unit normal on $\partial\Omega$ and the sign depends on whether the edges are ordered clockwise or counter-clockwise around the boundary.  
It is not necessary to include case $k=|V_{ext}^M| (=|E_{ext}^M|)$ here because the boundary data must already satisfy the compatability condition 
\[
\nabla \cdot \bChar{u}_h = 0 \Rightarrow \int_{\partial \Omega} \bChar{G}_h\cdot \bChar{n} \, ds =0.
\]
We have $c_3^{(0)}=0$ by convention in~\thmref{thm:Vh}, resulting in a square, lower-triangular system of equations here that is trivial to solve.  

A localized construction of basis functions is possible using a reference map.  
The reference triangle is denoted by $\hat{T}^M$, which has vertices $\hat{\bChar{z}}_1 = (0,0)$, $\hat{\bChar{z}}_2 = (1,0)$ and $\hat{\bChar{z}}_3 = (0,1)$.  
Let $T_k^M\in \tau_{\bChar{z}}^M$.  
We denote by $F_k:\hat{T}^M \to T_k^M$ an affine map onto $T_k^M$ with Jacobian 
$J_k$.  
The reference triangle is then divided into six subtriangles, say $\hat{T}_i$ for $1\leq i \leq 6$.  
This is done by adding four points that are the preimages under $F_k$ of the singular nodes and incenter on $T_k^M$; see~\figref{fig:refMap}.  

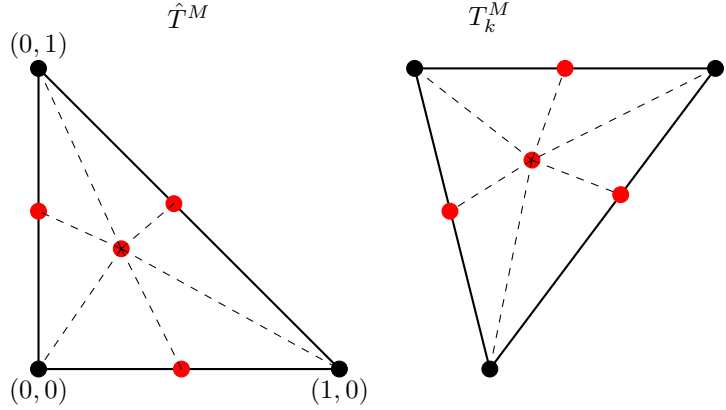
\begin{figure}
\centering 
\begin{tikzpicture}
\coordinate (Z1) at (0,0);
\coordinate (Z2) at (4,0);
\coordinate (Z3) at (0,4);
\draw [thick] (Z1) -- (Z2) -- (Z3) -- cycle;
\filldraw [black] (Z1) circle (3pt); 
\filldraw [black] (Z2) circle (3pt); 
\filldraw [black] (Z3) circle (3pt); 
\tkzLabelPoint[below](Z1){$(0,0)$}
\tkzLabelPoint[below](Z2){$(1,0)$}
\tkzLabelPoint[above](Z3){$(0,1)$}
\coordinate (Z12) at (1.9,0);
\coordinate (Z23) at (1.8,2.2);
\coordinate (Z30) at (0,2.1);
\coordinate (Zc) at (1.1,1.6);
\filldraw [red] (Z12) circle (3pt); 
\filldraw [red] (Z23) circle (3pt); 
\filldraw [red] (Z30) circle (3pt); 
\filldraw [red] (Zc) circle (3pt); 
\draw [dashed, black] (Zc) -- (Z1);
\draw [dashed, black] (Zc) -- (Z12);
\draw [dashed, black] (Zc) -- (Z2);
\draw [dashed, black] (Zc) -- (Z23);
\draw [dashed, black] (Zc) -- (Z3);
\draw [dashed, black] (Zc) -- (Z30);
\coordinate (B) at (6,0);
\coordinate (C) at (5,4);
\coordinate (E) at (9,4);
\filldraw [black] (B) circle (3pt); 
\filldraw [black] (C) circle (3pt); 
\filldraw [black] (E) circle (3pt); 
\coordinate (F) at (6.56,2.78);
\filldraw [red] (F) circle (3pt); 
\draw [thick] (B) -- (E) -- (C) -- cycle;
\draw [dashed, black] (E) -- (F);
\draw [dashed, black] (B) -- (F);
\draw [dashed, black] (C) -- (F);
\coordinate (ZTop) at (2,5);
\tkzLabelPoint[below](ZTop){$\hat{T}^M$}
\coordinate (Z2Top) at (6,5);
\tkzLabelPoint[below](Z2Top){$T_k^M$}
\coordinate (BE) at (7.74,2.32);
\coordinate (CE) at (7,4);
\coordinate (BC) at (5.47,2.1);
\draw [dashed, black] (F) -- (BC);
\draw [dashed, black] (F) -- (BE);
\draw [dashed, black] (F) -- (CE);
\filldraw [red] (BC) circle (3pt); 
\filldraw [red] (CE) circle (3pt); 
\filldraw [red] (BE) circle (3pt); 
\end{tikzpicture}
\caption{A localized construction of basis functions can be made explicit using a reference mapping $F_k : \hat{T}^M\to T_k^M$.  The mapping $F_k$ is defined as an  affine function that maps the vertices of $\hat{T}^M$ to those of $T_k^M$ (black). 
 The additional points used to subdivide $\hat{T}^M$ are the preimages under $F_k$ of the singular vertices and incenter on $T_k^M$ (red).}\label{fig:refMap}
\end{figure}

Functions $\bChar{v}_h\in \bChar{V}_h$ will be represented locally on $T_k^M$ by leveraging the Piola transformation
\begin{equation}
    \bChar{v}_h = \frac{1}{|\text{det}(J_k)|} J_k \hat{\bChar{v}}_h . 
   \label{eqn:Piola}
\end{equation}
As shown in the proof of~\lemref{lem:localBasis}, $\mathbf{v}_h$ can be determined uniquely on $T_k^M$ via three degrees of freedom using the basis~\eqref{eqn:Phi12}-\eqref{eqn:Phi3}.  
Given that $F_k$ maps reference vertex $\hat{\bChar{z}}_j$ to $\bChar{z}$, we can associate three basis functions on the reference element, say $\hat{\bSym{\Phi}}_i^{(j)}$, $1\leq i\leq 3$, such that the divergence (in reference coordinates) is $\hat{\nabla} \cdot \hat{\bSym{\Phi}}_i^{(j)} =0$, with 
\begin{align}
    \hat{\bSym{\Phi}}_i^{(j)} (\hat{\bChar{z}}_j) &= (\hat{\alpha}_i , \hat{\beta}_i ) \label{eqn:hatPhi12} \\ 
    \text{and} \quad 
    \int_{\hat{e}_j^M} \hat{\bSym{\Phi}}_i^{(j)} \cdot \hat{\bChar{n}}_j \, d\hat{s} &= \hat{\delta}_i . \label{eqn:hatPhi3} 
\end{align}
Here, $\hat{e}_j^M$ is the edge of $\hat{T}^M$ between $\hat{\bChar{z}}_{j-1}$ and $\hat{\bChar{z}}_{j}$ (modulo $3$), and $\hat{\bChar{n}}_j$ is the outward-pointing unit normal vector along that edge (see~\figref{fig:macroLabels}).  
We choose $(\hat{\alpha}_1 , \hat{\beta}_1 , \hat{\delta}_1) = (1,0,0)$, 
$(\hat{\alpha}_2 , \hat{\beta}_2 , \hat{\delta}_2) = (0,1,0)$, and 
$(\hat{\alpha}_3 , \hat{\beta}_3 , \hat{\delta}_3) = (0,0,1)$.  

Since there are three such functions for each reference vertex, we have a total 
of nine local basis functions on $\hat{T}^M$.  
These correspond to the three vertices of $T_k^M$.  
However, care is needed to map between the basis functions when changing variables.  Given the values 
\begin{equation*}
 (a_i , b_i)  = |\text{det}(J_k)| (J_k)^{-1} \bSym{\Phi}_i (\bChar{z}), \ i = 1,\, 2, 
\end{equation*} 
the correct mappings are 
\begin{align*}
     \bSym{\Phi}_1 |_{T_k^M} &= \frac{a_1}{|\text{det}(J_k)|} J_k \hat{\bSym{\Phi}}_1^{(j)} +\frac{b_1}{|\text{det}(J_k)|} J_k \hat{\bSym{\Phi}}_2^{(j)}  , \\ 
          \bSym{\Phi}_2 |_{T_k^M} &= \frac{a_2}{|\text{det}(J_k)|} J_k \hat{\bSym{\Phi}}_1^{(j)} +\frac{b_2}{|\text{det}(J_k)|} J_k \hat{\bSym{\Phi}}_2^{(j)}  , \\ 
          \text{and} \ 
          \bSym{\Phi}_3 |_{T_k^M} &= \frac{1}{|\text{det}(J_k)|} J_k \hat{\bSym{\Phi}}_3^{(j)} .
\end{align*}
The Piola transformation is well-known to preserve the divergence and the integral of the normal component along edges, so it is easy to verify the relationships above.  

Let $\hat{\bChar{z}}_i^s$, $1\leq i \leq 3$, satisfy the requirements that $\hat{\bChar{z}}_i^s$ lies on $\hat{e}_i^M$ and $\hat{\bChar{z}}_i^s = (F_k)^{-1} (\bChar{z}_i^s)$ for some singular vertex $\bChar{z}_i^s\in S$ that lies on an edge of $T_k^M$.  
Also, let $\hat{\bChar{z}}^c = (F_k)^{-1} (\bChar{z}^c)$, where $\bChar{z}^c$ is the incenter for $T_k^M$.  
Labels for points on $\hat{T}^M$ are shown in~\figref{fig:macroLabels}.  
In order to describe each $\hat{\bSym{\Phi}}_i^{(j)}$ explicitly, it suffices to list its values at the seven points $\hat{\bChar{z}}_i$ and $\hat{\bChar{z}}_i^s$, $1\leq i\leq 3$, and $\hat{\bChar{z}}^c$.  
Given that $\hat{\bSym{\Phi}}_i^{(j)} = (\hat{u}_i^{(j)},\hat{v}_i^{(j)})$, $\hat{\bChar{z}}^c = (\hat{x}^c , \hat{y}^c)$ and $\hat{\bChar{z}}_i^s = (\hat{x}_i^s , \hat{y}_i^s )$, these values are provided in~\tabref{tab:phi1} ($j=1$),~\tabref{tab:phi2} ($j=2$), and~\tabref{tab:phi3} ($j=3$). 

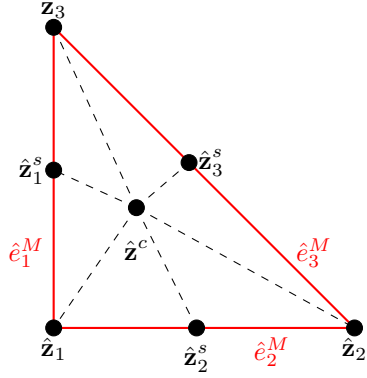
\begin{figure}
\centering 
\begin{tikzpicture}
\coordinate (Z1) at (0,0);
\coordinate (Z2) at (4,0);
\coordinate (Z3) at (0,4);
\draw [thick, red] (Z1) -- (Z2) -- (Z3) -- cycle;
\filldraw [black] (Z1) circle (3pt); 
\filldraw [black] (Z2) circle (3pt); 
\filldraw [black] (Z3) circle (3pt); 
\tkzLabelPoint[below](Z1){$\hat{\bChar{z}}_1$}
\tkzLabelPoint[below](Z2){$\hat{\bChar{z}}_2$}
\tkzLabelPoint[above](Z3){$\hat{\bChar{z}}_3$}
\coordinate (Z12) at (1.9,0);
\coordinate (Z23) at (1.8,2.2);
\tkzLabelPoint[above, right](Z23){$\hat{\bChar{z}}_3^s$}
\coordinate (Z30) at (0,2.1);
\tkzLabelPoint[left](Z30){$\hat{\bChar{z}}_1^s$}
\coordinate (Zc) at (1.1,1.6);
\filldraw [black] (Z12) circle (3pt); 
\filldraw [black] (Z23) circle (3pt); 
\filldraw [black] (Z30) circle (3pt); 
\filldraw [black] (Zc) circle (3pt); 
\coordinate (E1) at (0,1.0);
\tkzLabelPoint[left, red](E1){$\hat{e}_1^M$}
\coordinate (E2) at (2.9,0);
\tkzLabelPoint[below, red](E2){$\hat{e}_2^M$}
\coordinate (E3) at (3.1,1.0);
\tkzLabelPoint[above, right, red](E3){$\hat{e}_3^M$}
\draw [dashed, black] (Zc) -- (Z1);
\draw [dashed, black] (Zc) -- (Z12);
\draw [dashed, black] (Zc) -- (Z2);
\draw [dashed, black] (Zc) -- (Z23);
\draw [dashed, black] (Zc) -- (Z3);
\draw [dashed, black] (Zc) -- (Z30);
\coordinate (ZcLabel) at (1.1,1.3);
\tkzLabelPoint[below](ZcLabel){$\hat{\bChar{z}}^c$}
\coordinate (Z12Label) at (1.9,-0.1);
\tkzLabelPoint[below](Z12Label){$\hat{\bChar{z}}_2^s$}
\end{tikzpicture}
\caption{Labels for some edges (red) and points on $\hat{T}^M$.}\label{fig:macroLabels}
\end{figure}


\begin{table}
\centering 
    \begin{tabular}{c|ccccccc}
{} & $\hat{\bChar{z}}^c$ & $\hat{\bChar{z}}_1^s$ & $\hat{\bChar{z}}_2^s$ & $\hat{\bChar{z}}_3^s$ & $\hat{\bChar{z}}_1$ &  $\hat{\bChar{z}}_2$ &  $\hat{\bChar{z}}_3$ \\\hline
$\hat{u}_1^{(1)}$ & $-\hat{y}^c$ & $-\hat{y}_1^s$ & $1-\hat{x}_2^s$ & 0 & 1 & 0 & 0 \\ 
$\hat{v}_1^{(1)}$ & $\hat{y}^c$ & $\frac{\hat{y}_1^s -\hat{y}^c}{\hat{x}^c}$ & 0 & 0 & 0 & 0 & 0 \\ 
$\hat{u}_2^{(1)}$ & $\hat{x}^c$ & 0 & $\frac{\hat{x}_2^s -\hat{x}^c}{\hat{y}^c}$ & 0 & 0 & 0 & 0 \\ 
$\hat{v}_2^{(1)}$ & $-\hat{x}^c$ & $1-\hat{y}_1^s$ & $-\hat{x}_2^s$ & 0 & 1 & 0 & 0 \\ 
$\hat{u}_3^{(1)}$ & -2 & -2 & 2$\frac{\hat{x}^c -\hat{x}_2^s}{\hat{y}^c}$ & 0 & 0 & 0 & 0 \\ 
$\hat{v}_3^{(1)}$ & 2 & 2$\frac{\hat{y}_1^s -\hat{y}^c}{\hat{x}^c}$ & 2 & 0 & 0 & 0 & 0
    \end{tabular}
    \caption{\label{tab:phi1} Tabulated values of  $(\hat{u}_i^{(1)},\hat{v}_i^{(1)})=\hat{\bSym{\Phi}}_i^{(1)}$, $1\leq i\leq 3$, at the points shown in~\figref{fig:macroLabels}.}
\end{table}

\begin{table}
\centering 
    \begin{tabular}{c|ccccccc}
{} & $\hat{\bChar{z}}^c$ & $\hat{\bChar{z}}_1^s$ & $\hat{\bChar{z}}_2^s$ & $\hat{\bChar{z}}_3^s$ & $\hat{\bChar{z}}_1$ &  $\hat{\bChar{z}}_2$ &  $\hat{\bChar{z}}_3$ \\\hline
$\hat{u}_1^{(1)}$ & 0 & 0 & $\hat{x}_2^s$ & $\hat{x}_3^s+\frac{\hat{x}^c - \hat{x}_3^s}{1-\hat{x}^c -\hat{y}^c}$ & 0 & 1 & 0  \\ 
$\hat{v}_1^{(1)}$ & $-\hat{y}^c$ & 0 & 0 & $\frac{\hat{y}^c - \hat{y}_3^s}{1-\hat{x}^c -\hat{y}^c}$ & 0 & 0 & 0  \\ \
$\hat{u}_2^{(1)}$ & 0 & 0 & $\frac{\hat{x}_2^s -\hat{x}^c}{\hat{y}^c}$ & $\frac{\hat{x}^c-\hat{x}_3^s}{1-\hat{x}^c -\hat{y}^c}$ & 0 & 0 & 0  \\ 
$\hat{v}_2^{(1)}$ & $\hat{x}^c - 1$ & 0 & $\hat{x}_2^s -1$ & $\hat{x}_3^s+\frac{\hat{y}^c-\hat{y}_3^s}{1-\hat{x}^c -\hat{y}^c}$ & 0 & 1 & 0  \\ 
$\hat{u}_3^{(1)}$ & 0 & 0 & 2$\frac{\hat{x}_2^s -\hat{x}^c}{\hat{y}^c}$ & 2$\frac{\hat{x}^c-\hat{x}_3^s}{1-\hat{x}^c -\hat{y}^c}$ & 0 & 0 & 0  \\ 
$\hat{v}_3^{(1)}$ & -2 & 0 & -2 & 2$\frac{\hat{y}^c-\hat{y}_3^s}{1-\hat{x}^c -\hat{y}^c}$ & 0 & 0 & 0  
    \end{tabular}
    \caption{\label{tab:phi2} Tabulated values of  $(\hat{u}_i^{(2)},\hat{v}_i^{(2)})=\hat{\bSym{\Phi}}_i^{(2)}$, $1\leq i\leq 3$, at the points shown in~\figref{fig:macroLabels}.}
\end{table}

\begin{table}
\centering 
    \begin{tabular}{c|ccccccc}
{} & $\hat{\bChar{z}}^c$ & $\hat{\bChar{z}}_1^s$ & $\hat{\bChar{z}}_2^s$ & $\hat{\bChar{z}}_3^s$ & $\hat{\bChar{z}}_1$ &  $\hat{\bChar{z}}_2$ &  $\hat{\bChar{z}}_3$ \\\hline
$\hat{u}_1^{(1)}$ & $\hat{y}^c -1$ & $\hat{y}_1^s -1$ & 0 & $\frac{\hat{y}_3^s -\hat{y}^c}{1-\hat{x}^c -\hat{y}^c} -\hat{x}_3^s$ & 0 & 0 & 1  \\ 
$\hat{v}_1^{(1)}$ & 0 & $\frac{\hat{y}_1^s -\hat{y}^c}{\hat{x}^c}$ & 0 & $\frac{\hat{y}^c -\hat{y}_3^s}{1-\hat{x}^c -\hat{y}^c}$ & 0 & 0 & 0  \\ \
$\hat{u}_2^{(1)}$ & $-\hat{x}^c$ & 0 & 0 & $\frac{\hat{x}^c -\hat{x}_3^s}{1-\hat{x}^c -\hat{y}^c}$ & 0 & 0 & 0  \\ 
$\hat{v}_2^{(1)}$ & 0 & $\hat{y}_1^s$ & 0 & $\frac{\hat{x}_3^s -\hat{x}^c}{1-\hat{x}^c -\hat{y}^c}-\hat{x}_3^s$ & 0 & 0 & 1  \\ 
$\hat{u}_3^{(1)}$ & 2 & 2 & 0 & 2$\frac{\hat{x}_3^s -\hat{x}^c}{1-\hat{x}^c -\hat{y}^c}$ & 0 & 0 & 0  \\ 
$\hat{v}_3^{(1)}$ & 0 & 2$\frac{\hat{y}^c-\hat{y}_1^s}{\hat{x}^c}$ & 0 & 2$\frac{\hat{y}_3^s -\hat{y}^c}{1-\hat{x}^c -\hat{y}^c}$ & 0 & 0 & 0  
    \end{tabular}
    \caption{\label{tab:phi3} Tabulated values of  $(\hat{u}_i^{(3)},\hat{v}_i^{(3)})=\hat{\bSym{\Phi}}_i^{(3)}$, $1\leq i\leq 3$, at the points shown in~\figref{fig:macroLabels}.}
\end{table}

While the construction has been localized to $T_k^M$, it is important to note that the basis functions $\hat{\bSym{\Phi}}_i^{(j)}$ depend on the locations of the singular vertices and incenter of $T_k^M$.   
In implementation, assembly may be performed by looping over triangles in $\tau_h^M$ with a subloop over the six subtriangles of the Powell-Sabin split, computing nodal values locally for all basis functions.  
In turn, once these nodal values are found, each $\bSym{\Phi}_i$ has an obvious representation in the standard nodal basis used in continuous, piecewise linear finite element methods, so integration on each triangle can leverage classical techniques.  
Also, since the values of $\bSym{\Phi}_i$ at $\bChar{z}$ are defined already by~\lemref{lem:localBasis}, it is only necessary to compute basis values at incenters and singular vertices.  
Values at singular vertices can be stored to avoid recomputation when looping over triangles in $\tau_h^M$.    
The ultimate cost to assemble the linear system and solve is investigated in~\secref{sec:comps} via comparison with the classical saddle-point approach, which does not use the solenoidal basis.

\section{Computation of pressure from velocity}\label{sec:pressure} 
It is possible to compute $p_h$ after $\bChar{u}_h$ has been computed, if desired, by solving the problem posed in~\eqref{eqn:dfPressure}. 
Bases for $\bChar{X}_h^0\setminus \bChar{V}_h^0$ and $P_h$ will be constructed for this purpose in this section.  
The next result is used later to verify the correct number of basis functions, and also shows that the linear system for the pressure solve is square.  
\begin{lem}\label{lem:dimPh}
It holds that $\dim(\bChar{X}_h^0\setminus\bChar{V}_h^0) = \dim(P_h) = 2|\tau_h^M| + 2|E_{int}^M|-|V_i^M|$.
\end{lem}
\begin{proof}
The pressure space is piecewise constant on $\tau_h$, but with the constraints imposed by~\lemref{lem:discDiv}, which remove a degree of freedom for each singular point.    
As proved in~\cite{Fabien_2022}, an equivalent way to count is three basis functions per interior singular point and one per exterior singular point.  
Along with the constraint $P_h\subset L^2_0 (\Omega)$, it follows from the correspondence between singular points and edges of $\tau_h^M$ that $\dim(P_h) = 3|E_{int}^M| +|E_{ext}^M|-1$.  
We count three edges per triangle, so including multiplicity for interior edges we find $3|\tau_h^M| = 2|E_{int}^M| +|E_{ext}^M|$, thus $\dim(P_h) = 3|\tau_h^M| +|E_{int}^M|-1$. 
Insert the Euler identity in the form 
\[
1 = |\tau_h^M| +|V^M| - |E^M| = |\tau_h^M| +|V_i^M| - |E_{int}^M|
\]
and we find $\dim(P_h) = 2|\tau_h^M| + 2|E_{int}^M|-|V_i^M|$.  

Now refer to~\thmref{thm:Vh}: $|\mathcal{B}^0| = \dim(\bChar{V}_h^0) = 3|V_i^M|$.  
All interior vertices on the Powell-Sabin split contribute toward the dimension of $\bChar{X}_h^0$; we may group the vertices as $V_i^M$ plus interior singular vertices and incenters of triangles in $\tau_h^M$, so that 
\begin{multline*} 
\dim(\bChar{X}_h^0)=2(|V_i^M| +|E_{int}^M| +|\tau_h^M|) \\
\Rightarrow 
\dim(\bChar{X}_h^0 \setminus \bChar{V}_h^0)=  2|\tau_h^M| + 2|E_{int}^M|-|V_i^M| =\dim(P_h) .
\end{multline*}
\end{proof}

The construction of basis functions builds upon the scalar finite element space 
\[
S_h = \left\{ \lambda \in \mathcal{C} (\overline{\Omega}) \, | \, \lambda|_T \in \mathbb{P}_1 (T),\ \forall T\in\tau_h \right\} .
\]
Let $\{\lambda_\bChar{z}\}$ for all vertices $\bChar{z}\in V$ be the usual Lagrangian (nodal) basis for $S_h$, such that for a given indexing $\{ \bChar{z}_i\}$ of $V$ we have $\lambda_{\bChar{z}_j} (\bChar{z}_k)=\delta_{jk}$.  
Next, we denote the singular points on the edges $e_j^M\in E_{int}^M$ as $\bChar{z}_s^{(j)}$ for $1\leq j \leq |E_{int}^M|$.  
On each edge $e_j^M$, associate a unit normal vector $\bChar{n}_j$ and a unit tangent vector $\bChar{t}_j$.  
The basis construction will make use of some functions 
\begin{equation*}
    \bSym{\psi}_s^{(j)} = \lambda_{\bChar{z}_s^{(j)}}\, \bChar{n}_j \quad 
\text{and} \quad \bSym{\phi}_s^{(j)} = \lambda_{\bChar{z}_s^{(j)}}\, \bChar{t}_j .
\end{equation*} 
The support of these functions is the set of four triangles $T\in\tau_h$ with the common interior singular vertex $\bChar{z}_s^{(j)}$.  
If we denote the incenters of each triangle $T_k^M\in \tau_h^M$ by $\bChar{z}_c^{(k)}$ for $1\leq k\leq |\tau_h^M|$, then the following are also needed: 
\begin{align*}
    \bSym{\psi}_c^{(k)} = \lambda_{\bChar{z}_c^{(k)}}\, (1,0) \quad 
\text{and}\quad  \bSym{\phi}_c^{(k)} = \lambda_{\bChar{z}_c^{(k)}}\, (0,1).
\end{align*} 
The support of these functions is the set of six triangles $T\in\tau_h$ with the common vertex $\bChar{z}_c^{(k)}$.  

Some functions above must be excluded to extract a basis.  
This must be done carefully.  
First, identify one vertex $\bChar{z}_0\in V_{ext}^M$ (that lies on $\partial\Omega$; see also~\thmref{thm:Vh}), and group it together with the interior vertices $\bChar{v}_j \in V_i^M$, $1\leq j \leq |V_i^M|$.  
There exists some subset of edges in $E_{int}^M$, say $e_k^M$ that we may index (without loss of generality) as $1\leq k \leq N$, such that these edges form a spanning tree for the graph created by the vertex set $\{ \bChar{z}_0 , \ldots , \bChar{z}_{|V_i^M|} \}$ and all of the edges that have both endpoints $\bSym{\zeta}_1, \bSym{\zeta}_2 \in \{ \bChar{z}_0 , \ldots , \bChar{z}_{|V_i^M|} \}$. A spanning tree of a connected graph is itself a connected subgraph that touches every vertex in the graph without creating any cycles.  These particular vertices and edges form a finite connected graph, so a spanning tree exists. 
It is well-known that $N = |V_i^M|$.  We compute a spanning tree using Kruskal's algorithm \cite{kruskal1956} for implementation. 
\begin{thm}\label{thm:Ph}
    Given the indexing above, the following set $\mathcal{S}$ is a basis of $\bChar{X}_h^0 \setminus \bChar{V}_h^0$: 
    \[
\mathcal{S} = \left\{ \bSym{\psi}_s^{(j)} \right\}_{j=|V_i^M|+1}^{|E_{int}^M|}\bigcup \left\{ \bSym{\phi}_s^{(j)} \right\}_{j=1}^{|E_{int}^M|} \bigcup \left\{ \bSym{\psi}_c^{(k)} \right\}_{k=1}^{|\tau_h^M|} \bigcup \left\{ \bSym{\phi}_c^{(k)} \right\}_{k=1}^{|\tau_h^M|} .
    \]
Furthermore, $\nabla \cdot \mathcal{S}$ is a basis of $P_h$.  
\end{thm}
\begin{proof}
Indeed, by~\lemref{lem:dimPh} we have the correct number of basis functions.  
Note that the functions in $\mathcal{S}$ are equivalent to a subset of the standard nodal basis functions typically used for the globally continuous, piecewise linear velocity space on $\tau_h$ except that at singular points the vectors are expressed in the local basis $\{ \bChar{n}_j , \bChar{t}_j \}$ of $\reals^2$.  
Since the properties of these basis functions are well-known, it is immediately clear that $\mathcal{S}\subset \bChar{X}_h^0\setminus\bChar{V}_h^0$ by construction and that $\mathcal{S}$ is linearly independent.   Therefore, it is a basis for $\bChar{X}_h^0\setminus\bChar{V}_h^0$. 

Note that each function in $\nabla \cdot \mathcal{S}$ is also in $L^2_0(\Omega)$ (apply the divergence form of Green's theorem), thus each function is in $P_h$ by~\lemref{lem:discDiv}. 
It remains to show that if 
\begin{multline*}
    \sum_{j=|V_i^M|+1}^{|E_{int}^M|} a_j\nabla\cdot\bSym{\psi}_e^{(j)} + \sum_{j=1}^{|E_{int}^M|} b_j\nabla\cdot\bSym{\phi}_e^{(j)} + \sum_{k=1}^{|\tau_h^M|} c_k\nabla\cdot\bSym{\psi}_c^{(k)} + \sum_{k=1}^{|\tau_h^M|} d_k\nabla\cdot\bSym{\phi}_c^{(k)} = 0, 
\end{multline*} 
then the coefficients shown are all zero.  
Equivalently, $0 = \nabla\cdot \bSym{\phi}$ where 
\begin{equation*}
\bSym{\phi} = \sum_{j=|V_i^M|+1}^{|E_{int}^M|} a_j\bSym{\psi}_e^{(j)} + \sum_{j=1}^{|E_{int}^M|} b_j\bSym{\phi}_e^{(j)} + \sum_{k=1}^{|\tau_h^M|} c_k\bSym{\psi}_c^{(k)} + \sum_{k=1}^{|\tau_h^M|} d_k\bSym{\phi}_c^{(k)}.
\end{equation*}
Thus $\bSym{\phi} \in \bChar{V}_h^0$.  
Since $\bSym{\phi}$ is already expressed as a linear combination of basis functions, it suffices to show that $\bSym{\phi} = \bChar{0}$.  
The basis functions are non-zero at (specific) singular points and incenters but vanish at neighboring vertices, so $\bSym{\phi}(\bChar{z}) = \bChar{0}$ for any $\bChar{z}\in V^M$. 
Upon changing to the basis of $\bChar{V}_h^0$ in~\thmref{thm:Vh}, this latter property implies 
\[
\bSym{\phi} = \sum_{k=1}^{|V_i^M|} \beta_ k \bSym{\Phi}_3^{(k)} .
\]
Given any edge $e_j^M$ with $1\leq j\leq |V_i^M|$, note that 
\begin{align*}
\int_{e_j^M} \bSym{\psi}_e^{(m)}\cdot \bChar{n}_j\, ds &= 0 , \ |V_i^M|+1 \leq m \leq |E_{int}^M|, \\ 
\int_{e_j^M} \bSym{\phi}_e^{(m)}\cdot \bChar{n}_j\, ds &= 0 , \ 1 \leq m \leq |E_{int}^M|, \\ 
\int_{e_j^M} \bSym{\psi}_c^{(m)}\cdot \bChar{n}_j\, ds &= 0 , \ 1 \leq m \leq |\tau_h^M|, \\ 
\int_{e_j^M} \bSym{\phi}_c^{(m)}\cdot \bChar{n}_j\, ds &= 0 , \ 1 \leq m \leq |\tau_h^M|, \\ 
\Rightarrow  \int_{e_j^M} \bSym{\phi}\cdot \bChar{n}_j\, ds &= 0 
\end{align*} 
In turn, we have a square, homogeneous linear system of equations for the coefficients $\beta_k$: 
\begin{equation}
\sum_{k=1}^{|V_i^M|} \beta_ k \int_{e_j^M} \bSym{\Phi}_3^{(k)} \cdot \bChar{n}_j\, ds = \sum_{k=1}^{|V_i^M|} \beta_ k  \left( \pm \delta_{kj} \right)= 0, \ 1\leq j \leq |V_i^M|.  \label{eqn:thm2_1} 
\end{equation}
Recall that the edges $e_j^M$ form a spanning tree for the vertices $\{ \bChar{z}_0 , \ldots , \bChar{z}_{|V_i^M|} \}$.  
Each basis function $\bSym{\Phi}_3^{(k)}$ is associated with vertex $\bChar{z}_k$, and they can be put into one-to-one correspondence with the edges $e_j^M$, $1\leq j \leq |V_i^M|$, say $\bChar{z}_k$ is an endpoint of edge $e_{j_k}^M$.  
Since the labeling of interior vertices is arbitrary, for convenience we may assume that edge $e_{j_k}^M$ has endpoints $\bChar{z}_{k-1}$ and $\bChar{z}_k$.  
As $\bChar{z}_0$ lies on $\partial\Omega$, only basis function $\bSym{\Phi}_3^{(1)}$ is supported on edge $e_{j_1}^M$, so $\beta_1=0$ by~\eqref{eqn:thm2_1}.     
Inductively, given that $\beta_{k-1}=0$ we find that only $\bSym{\Phi}_3^{(k)}$ is supported on edge $e_{j_k}^M$, so $\beta_k=0$ also for $2\leq k\leq |V_i^M|$.  
Thus, $\bSym{\phi} = \bChar{0}$, completing the proof.  
\end{proof}

\section{Computations}\label{sec:comps} 
On the domain $\Omega = (0,1) \times (0,1) \subset \reals^2$ we will use the manufactured solution
\begin{equation}\label{manusoln}\bChar{u} = \begin{pmatrix} \sin(x)\cos(y) \\ -\cos(x)\sin(y) \end{pmatrix}, \quad p = xy - \frac{1}{4} 
\end{equation}
with parameter choice $\nu=1$.  The usual norms for the spaces $L^2 (\Omega)$ and $[H^1 (\Omega)]^2$ are denoted by $\| \cdot \|$ and $\|\cdot \|_1$, respectively.
Computations that use the solenoidal basis for velocity are labelled as SOL, as opposed to SP for the saddle-point formulation (without the solenoidal basis). 
All computations are done in MATLAB using the multifrontal direct solver package invoked by the command \texttt{mldivide}, commonly referred to as the ``backslash'' operator.  
The Choleski solver is called by \texttt{mldivide} for all SOL calculations in order to exploit the symmetric, positive definite matrix properties, whereas for the SP system the solver called is MA57.  The MA57 solver is optimized for sparse, symmetric indefinite systems. The macro-meshes $\tau_h^M$ were generated using code from \cite{gockenbach,yonkermesh}.  

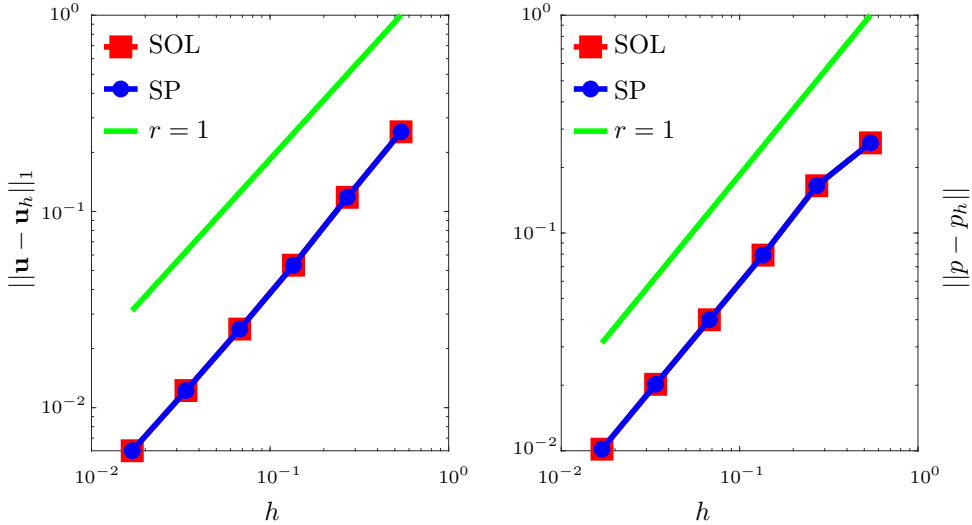
\begin{figure}
\centering
\begin{overpic}[width=11.0cm]{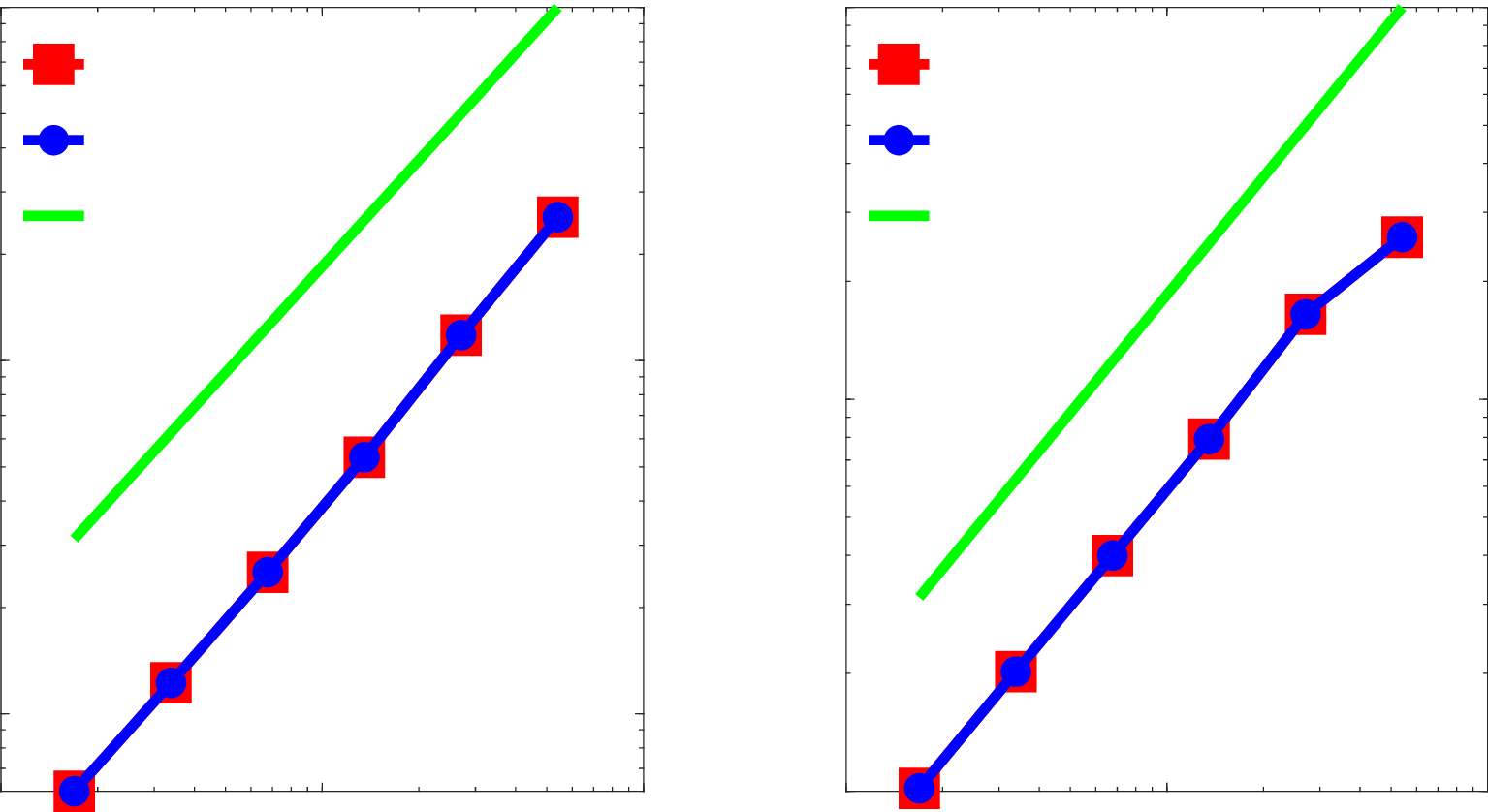}
\put(21,-7){$h$}
\put(-10,21){\rotatebox{90}{$||\bChar{u}-\bChar{u}_h||_1$}}
\put(-2,-2.5){\scriptsize $10^{-2}$}
\put(20,-2.5){\scriptsize $10^{-1}$}
\put(42,-2.5){\scriptsize $10^{0}$}
\put(73,-7){$h$}
\put(54,-2.5){\scriptsize $10^{-2}$}
\put(76,-2.5){\scriptsize $10^{-1}$}
\put(99,-2.5){\scriptsize $10^{0}$}
\put(-6,6){\scriptsize $10^{-2}$}
\put(51,1){\scriptsize $10^{-2}$}
\put(-6,29.5){\scriptsize $10^{-1}$}
\put(51,27){\scriptsize $10^{-1}$}
\put(-5,53.5){\scriptsize $10^{0}$}
\put(52,53.5){\scriptsize $10^{0}$}
\put(7,49.5){SOL}
\put(7,44){SP}
\put(7,39.25){$r=1$}
\put(63.25,49.5){SOL}
\put(63.25,44){SP}
\put(63.25,39.25){$r=1$}
\put(103,21){\rotatebox{90}{$||p-p_h||$}}
\end{overpic}
\vspace{6mm}
\caption{Velocity and pressure errors converge with the optimal scaling $\| \bChar{u} - \bChar{u}_h\|_1 = \mathcal{O}(h^r)$ and $\| p - p_h\| = \mathcal{O}(h^r)$ with $r=1$.  The SOL and SP velocity calculations agree to round-off precision, so the plotted data overlaps.}\label{fig:Error}
\end{figure}

Errors for velocity and pressure are shown in~\figref{fig:Error} in their respective norms. The horizontal axis shows the mesh size $h$ for the Powell-Sabin splits $\tau_h$. The optimal convergence rates were observed for both velocity and pressure variables.  Between SOL and SP, the computed states agreed to within the limits of round-off error.
We also observed (not shown) that the condition number of the matrix for the SOL velocity system was consistently less than 1\% of the corresponding SP condition number. 

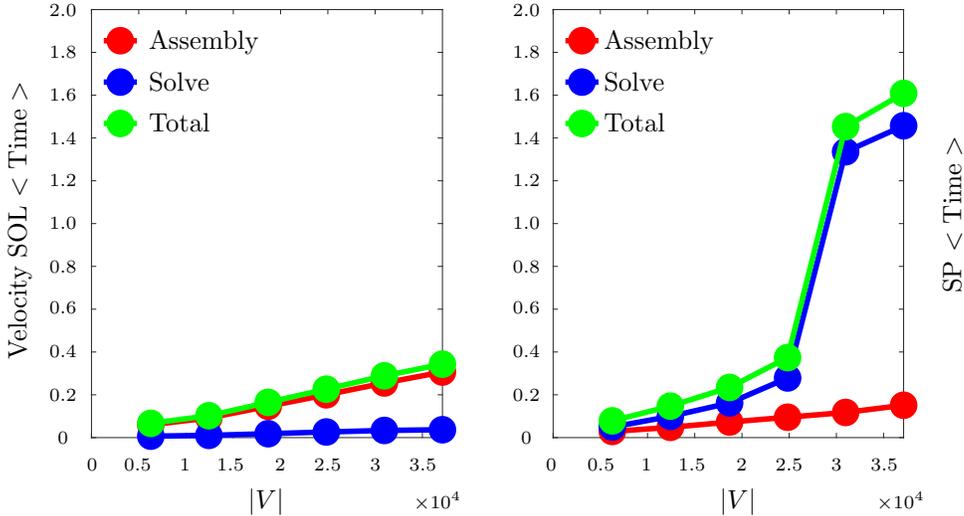
\begin{figure}
\centering
\begin{overpic}[width=11.0cm]{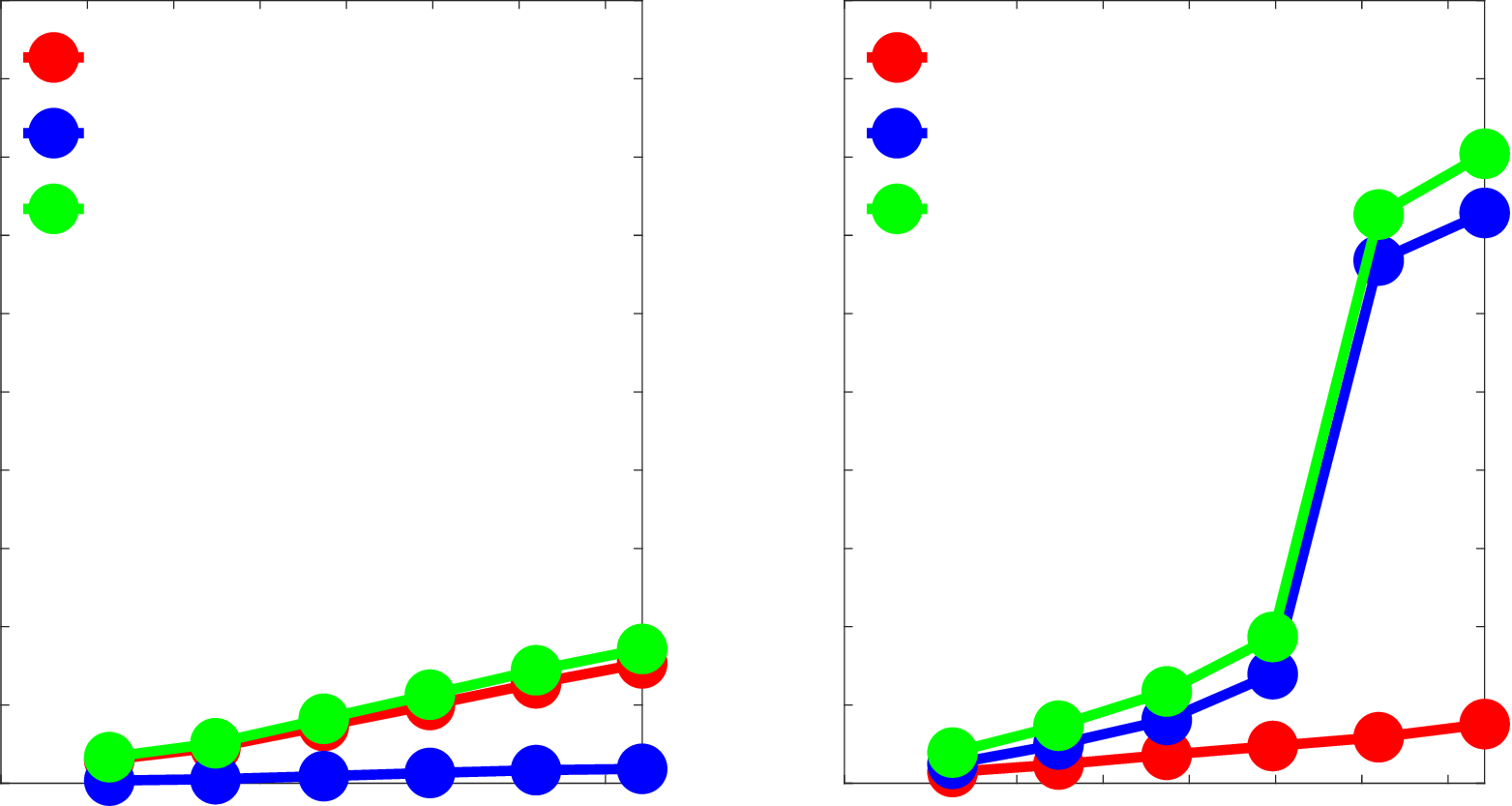}
\put(19,-7){$|V|$}
\put(-10,11){\rotatebox{90}{Velocity SOL $\Time$}}
\put(103,19){\rotatebox{90}{SP $\Time$}}
\put(76,-7){$|V|$}
\put(7,48.5){Assembly}
\put(7,43.5){Solve}
\put(7,38.5){Total}
\put(62,48.5){Assembly}
\put(62,43.5){Solve}
\put(62,38.5){Total}
\put(-0.5,-2.5){\scriptsize 0}
\put(4,-2.5){\scriptsize 0.5}
\put(10.75,-2.5){\scriptsize 1}
\put(15.25,-2.5){\scriptsize 1.5}
\put(22.25,-2.5){\scriptsize 2}
\put(27,-2.5){\scriptsize 2.5}
\put(33.75,-2.5){\scriptsize 3}
\put(38.5,-2.5){\scriptsize 3.5}
\put(39,-7){\scriptsize $\times 10^{4}$}
\put(55,-2.5){\scriptsize 0}
\put(60,-2.5){\scriptsize 0.5}
\put(66.5,-2.5){\scriptsize 1}
\put(71,-2.5){\scriptsize 1.5}
\put(78,-2.5){\scriptsize 2}
\put(82.75,-2.5){\scriptsize 2.5}
\put(89.5,-2.5){\scriptsize 3}
\put(94.25,-2.5){\scriptsize 3.5}
\put(95,-7){\scriptsize $\times 10^{4}$}
\put(-4,0.75){\scriptsize 0}
\put(-5,6){\scriptsize 0.2}
\put(-5,11.25){\scriptsize 0.4}
\put(-5,16.5){\scriptsize 0.6}
\put(-5,21.5){\scriptsize 0.8}
\put(-5,26.75){\scriptsize 1.0}
\put(-5,32){\scriptsize 1.2}
\put(-5,37){\scriptsize 1.4}
\put(-5,42.25){\scriptsize 1.6}
\put(-5,47.5){\scriptsize 1.8}
\put(-5,52.5){\scriptsize 2.0}
\put(52,0.75){\scriptsize 0}
\put(51,6){\scriptsize 0.2}
\put(51,11.25){\scriptsize 0.4}
\put(51,16.5){\scriptsize 0.6}
\put(51,21.5){\scriptsize 0.8}
\put(51,26.75){\scriptsize 1.0}
\put(51,32){\scriptsize 1.2}
\put(51,37){\scriptsize 1.4}
\put(51,42.25){\scriptsize 1.6}
\put(51,47.5){\scriptsize 1.8}
\put(51,52.5){\scriptsize 2.0}
\end{overpic}
\vspace{6mm}
\caption{Time trials comparing velocity computation with SOL versus SP.}
\label{fig:VelocitySOLvsSP}
\end{figure}

Run times were computed to assemble and to solve the systems.  
In~\figref{fig:VelocitySOLvsSP} we plot $\Time$ versus $|V|$ , where $|V|$ is the number of vertices in the Powell-Sabin mesh and $\Time$ is the averaged timing over 20 runs.
The timings for SOL are only for computing velocity.  
A jump up in solver time was observed on the two finest meshes for the SP case; output from the solver package indicated this was due to an increase in the cost to correctly factor the system.  
Although assembly times using SOL are significant, there is an overall time savings to compute velocity alone using SOL versus SP on fine meshes.  
Solver timings are much smaller for SOL even on fairly coarse meshes, which would translate to further efficiency gains for applications that require repeated solves without significant reassembly.  

\begin{figure}
\centering
\begin{overpic}[width=11.0cm]{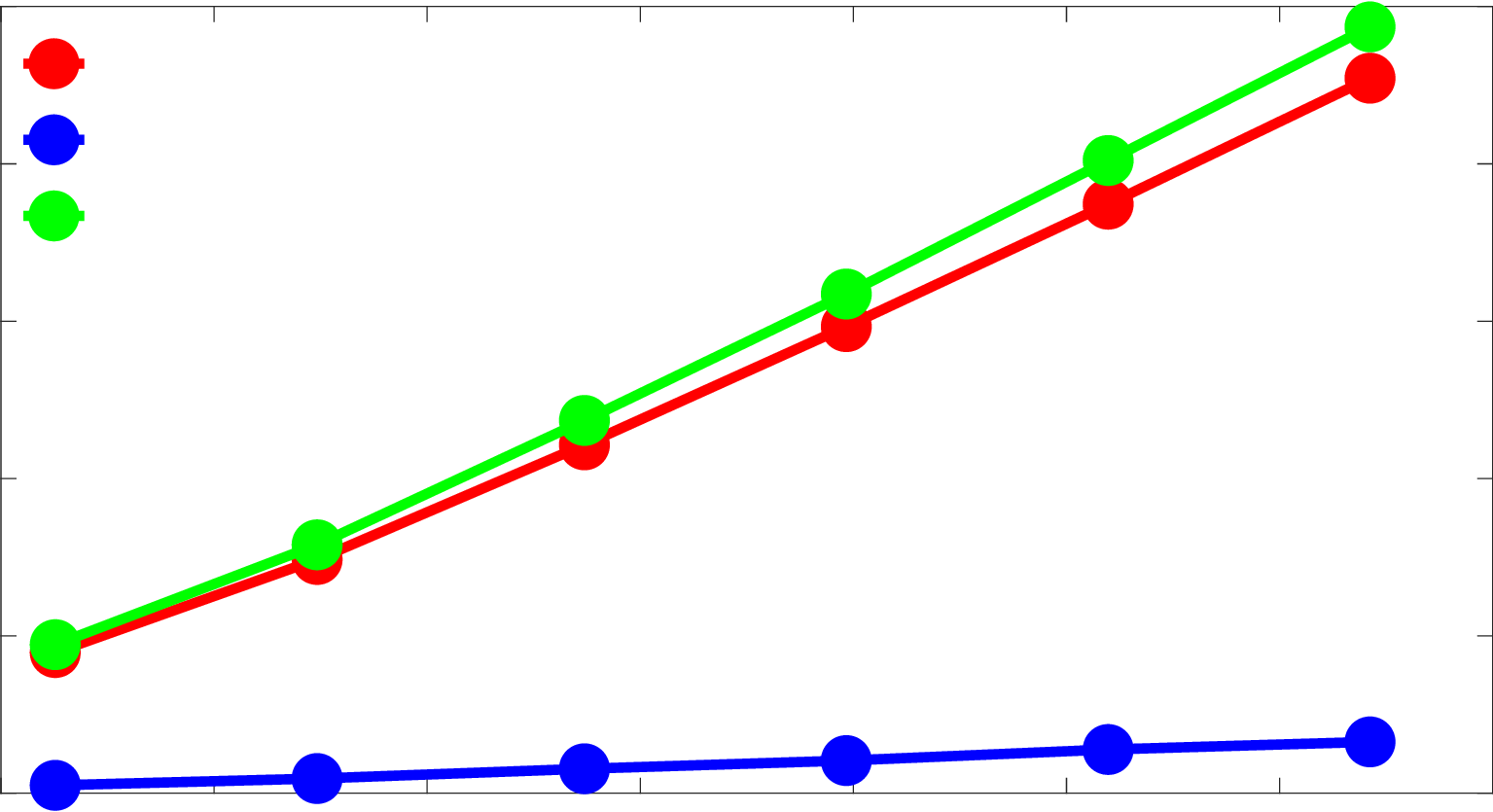}
\put(-10,12){\rotatebox{90}{Pressure SOL $\Time$}}
\put(50,-7){$|V|$}
\put(7,49){Assembly}
\put(7,44){Solve}
\put(7,39){Total}
\put(-3,0.5){\scriptsize 0}
\put(-5,11.25){\scriptsize 0.05}
\put(-4.5,21.75){\scriptsize 0.1}
\put(-5,32){\scriptsize 0.15}
\put(-4.5,43){\scriptsize 0.2}
\put(-5,53){\scriptsize 0.25}
\put(-2,-2.5){\scriptsize 0.5}
\put(13.75,-2.5){\scriptsize 1}
\put(27,-2.5){\scriptsize 1.5}
\put(42.25,-2.5){\scriptsize 2}
\put(55.5,-2.5){\scriptsize 2.5}
\put(70.75,-2.5){\scriptsize 3}
\put(84.25,-2.5){\scriptsize 3.5}
\put(99,-2.5){\scriptsize 4}
\put(99,-7){\scriptsize $\times 10^{4}$}
\end{overpic}
\vspace{6mm}
\caption{SOL Pressure Time}
\label{fig:PressureSOL}
\end{figure}

\begin{figure}
\centering
\begin{overpic}[width=11.0cm]{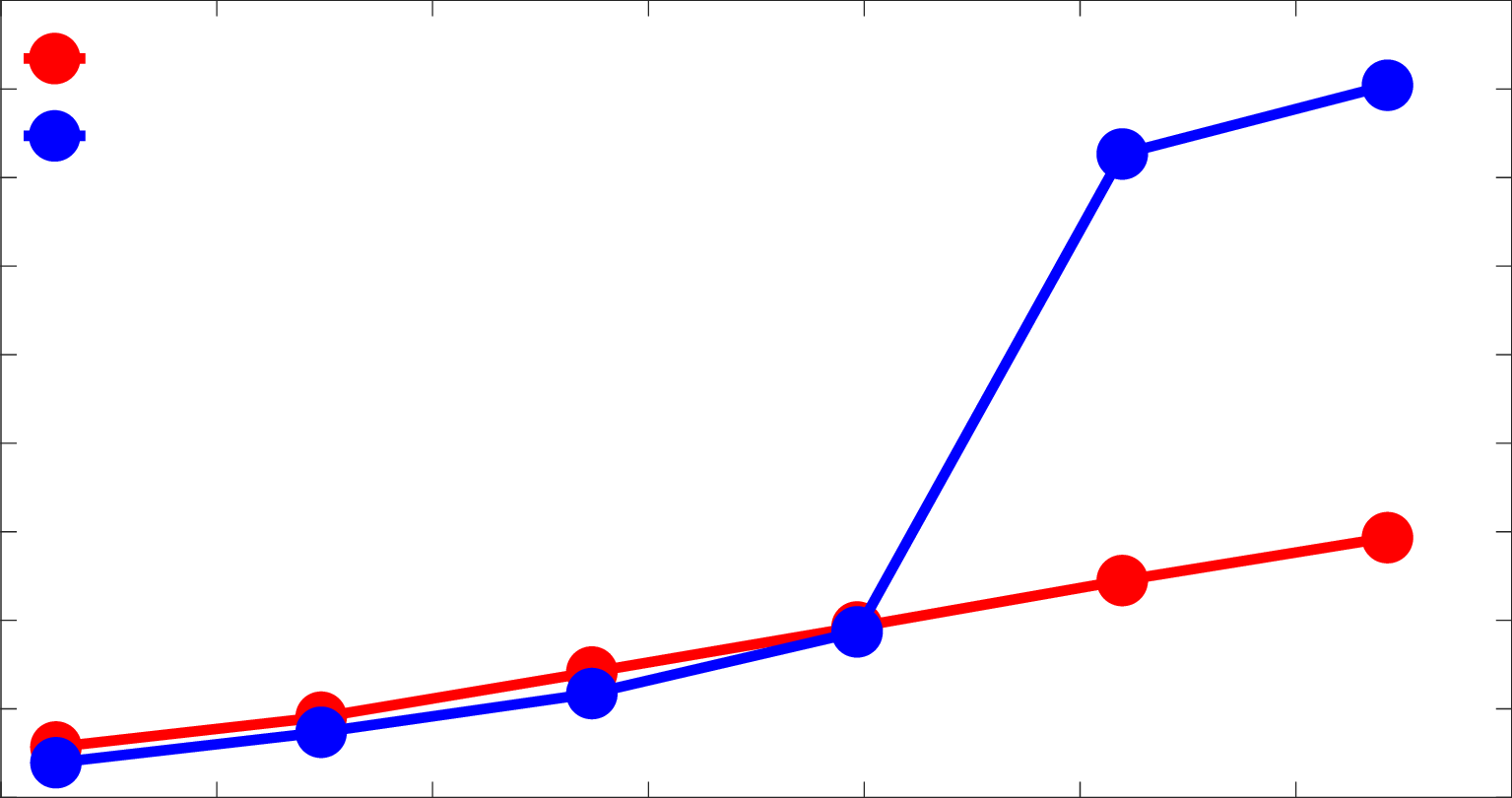}
\put(-10,16){\rotatebox{90}{Total $\Time$}}
\put(50,-7){$|V|$}
\put(7,48){SOL}
\put(7.5,42.5){SP}
\put(-4,-0.25){\scriptsize 0}
\put(-5,5.25){\scriptsize 0.2}
\put(-5,11){\scriptsize 0.4}
\put(-5,17){\scriptsize 0.6}
\put(-5,23){\scriptsize 0.8}
\put(-5,28.75){\scriptsize 1.0}
\put(-5,34.5){\scriptsize 1.2}
\put(-5,40.5){\scriptsize 1.4}
\put(-5,46.25){\scriptsize 1.6}
\put(-5,52){\scriptsize 1.8}
\put(-2,-3){\scriptsize 0.5}
\put(13.75,-3){\scriptsize 1}
\put(27,-3){\scriptsize 1.5}
\put(42.25,-3){\scriptsize 2}
\put(55.5,-3){\scriptsize 2.5}
\put(70.75,-3){\scriptsize 3}
\put(84.25,-3){\scriptsize 3.5}
\put(99,-3){\scriptsize 4}
\put(99,-7){\scriptsize $\times 10^{4}$}
\end{overpic}
\vspace{6mm}
\caption{Total time to assemble and solve for pressure and velocity.}
\label{fig:AllTime}
\end{figure}
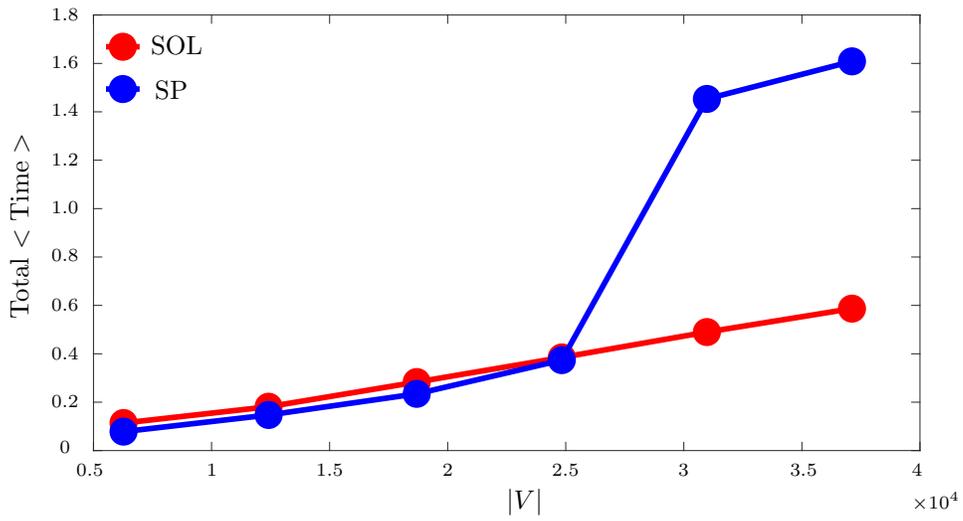

Timings are displayed for the computation of pressure via the SOL method in~\figref{fig:PressureSOL}.  
These times are somewhat less than the corresponding times for the SOL velocity computations.  
The total time required to compute both velocity and pressure is compared in~\figref{fig:AllTime}.  
In conclusion, the SOL method supports a more efficient system solve on fairly coarse meshes, and may be used to compute velocity alone more quickly.  
Due to assembly costs, the SOL and SP methods are about even in computing the full velocity-pressure system until the mesh is fine enough, when SOL is again faster. 

\section{Summary}\label{sec:summary}
A solenoidal basis of a first-order polynomial velocity space was explicitly constructed in this paper for the conforming, divergence-free velocity-pressure finite element pair studied in~\cite{Fabien_2022} for the Stokes problem.    
These elements are defined on the Powell-Sabin split of a triangulation for a two-dimensional domain.  
Inhomogeneous Dirichlet conditions were enabled by constructing an interpolation operator for boundary data into the trace of the solenoidal velocity subspace. 
A basis of the pressure space was also derived using divergences of non-solenoidal functions from the full discrete velocity space.  
Its construction used a graph-theoretic argument to discard functions associated with a certain subset of vertices in the mesh, retaining a spanning set.  
The implementation can be accomplished by generating a spanning tree related to the mesh, for which purpose their is an efficient algorithm due to Kruskal~\cite{kruskal1956}.  
All basis functions were shown to have local support.  
Rigorous proofs were included that verify the construction of bases.  
Also, some implementation details for the solenoidal velocity basis were discussed that enable a localization of the computations to assemble the linear system.  

The use of the bases studied enables a block-triangular structure for the velocity-pressure system, wherein the velocity can be isolated and computed without computing the pressure.  
The computed velocity can then be used after to compute the pressure, if desired.  
The non-solenoidal component of the discrete velocity space is eliminated from the velocity computation, which further reduces the problem size.  

The system matrices for the velocity and pressure solves are both symmetric, positive-definite and support the use of the Choleski direct solver.  
Computations using MATLAB verified a significant decrease in time to compute velocity alone compared to solving the usual saddle-point system.  
On fine enough meshes, even computing both pressure and velocity together was faster using the solenoidal velocity basis.  

Other benefits are expected for situations not included in the scope of this paper.  
The conjugate gradient method and associated preconditioners could be applied for larger sized problems when a direct solve is not feasible.  
Also, for applications that require repeated linear solves with a reduced cost of system reassembly compared to the initial assembly, there could be a larger efficiency gain using the solenoidal basis than was demonstrated in this paper.   

\bibliography{refs}{}
\bibliographystyle{plain} 

\end{document}